\documentclass[11pt]{amsart}
\usepackage[english]{babel}
\usepackage{amsmath}
\usepackage{amsthm}
\usepackage{amssymb}
\usepackage{mathrsfs}
\usepackage{enumerate}
\usepackage[notcite, final, notref]{showkeys}
\usepackage{dsfont}
\usepackage{color}
\usepackage{shadow}
\newtheorem{thm}{Theorem}[section]

\newtheorem{lem}[thm]{Lemma}
\newtheorem{prop}[thm]{Proposition}
\newtheorem{cor}[thm]{Corollary}
\newtheorem{defn}[thm]{Definition}

\theoremstyle{definition}
\newtheorem{rem}[thm]{Remark}
\newtheorem{ex}[thm]{Example}

\newcommand{\C}{\mathsf C}

\title[Polyanalytic approach to de Branges spaces and Schur analysis]{ON A POLYANALYTIC APPROACH TO NONCOMMUTATIVE DE
BRANGES-ROVNYAK SPACES AND SCHUR ANALYSIS}

\author[D. Alpay]{Daniel Alpay}
\address{(DA)
Faculty of Mathematics, Physics, and Computation\\
Schmid College of Science and Technology\\
Chapman University\\
One University Drive
Orange, California 92866\\
USA}
\email{alpay@chapman.edu}

\author[F. Colombo]{Fabrizio Colombo}
\address{(FC) Politecnico di
Milano\\Dipartimento di Matematica\\Via E. Bonardi, 9\\20133 Milano\\Italy}
\email{fabrizio.colombo@polimi.it}

\author[K. Diki]{Kamal Diki}
\address{(KD) Politecnico di
Milano\\Dipartimento di Matematica\\Via E. Bonardi, 9\\20133 Milano,
Italy}
\email{kamal.diki@polimi.it}

\author[I. Sabadini]{Irene Sabadini}
\address{(IS) Politecnico di
Milano\\Dipartimento di Matematica\\Via E. Bonardi, 9\\20133 Milano\\Italy}
\email{irene.sabadini@polimi.it}

\begin{document}
\maketitle
\begin{abstract}
  In this paper we begin the study of Schur analysis and of de Branges-Rovnyak spaces in the framework of Fueter hyperholomorphic functions. The difference with other approaches is that we consider the
  class of functions spanned by Appell-like polynomials. This approach is very efficient from various points of view, for example in operator theory, and allows us to make connections with the
  recently developed theory of slice polyanalytic functions.
  We tackle a number of problems: we describe a Hardy space, Schur multipliers and related results. We also discuss Blaschke functions, Herglotz  multipliers and their associated kernels and
  Hilbert spaces. Finally, we consider the counterpart of the half-space case, and the corresponding Hardy space, Schur multipliers and Carath\'eodory multipliers.
\end{abstract}


\noindent {\em }
\date{today}
\tableofcontents
\section{Introduction}

\setcounter{equation}{0}
Two important function theories that allow to extend complex analysis and operator theory results to higher dimensions are the so-called monogenic and slice monogenic functions with values in a Clifford algebra.  In the case of quaternions these two theories are known as Fueter hyperholomorphic and slice regular or slice hyperholomorphic functions, respectively, see \cite{bds, MR2089988, MR2752913, MR3013643, gurlebeck2008application}. For the necessary preliminaries on quaternions, we refer the reader to Section 2. An interesting problem  is to investigate the possible relations and intersections between these two different theories. We note that it is always possible to construct Fueter hyperholomorphic functions starting from slice regular ones using different techniques such as the Fueter mapping theorem \cite{MR2762317, MR2787441}, or using the Radon and dual Radon transforms, see \cite{CLSSRadon2015}. But in general, the slice monogenicity does not imply, nor is implied by monogenicity. However, in \cite{ADSP2019,ADSP2019bis} the authors extended the notion of slice regular functions to higher order by considering the so-called slice polyanalytic functions.  These functions can be considered from three different points of view. The first approach consists of viewing the space of quaternions $\mathbb H$ as union of complex planes and to see these functions as a subclass of null solutions of the $n$-th power of the Cauchy-Riemann operator with respect to each complex plane. The second approach is based on the so-called poly-decomposition which allows us to consider such functions as sums of the form $$\displaystyle\sum_{k=0}^{n-1}\overline{x}^kf_k(x),\qquad x\in\mathbb{H}$$
  with all the components $f_k$ which are slice regular functions and $n$ is the order of poly-analyticity. The third approach consists in considering slice polyanalytic functions as subclass of the
null solutions of the $n$-th power of a global operator with non-constant coefficients, see \cite{ADSP2020}. The study in this paper is in the quaternionic context and it is based on some polynomials
$(P_n(x))_{n\geq 0}$ where
$$
P_n(x)=\sum_{j=0}^n L_{j,n} \bar{x}^jx^{n-j}, \qquad n\geq 0,
$$
that are at the same time Fueter hyperholomorphic and slice polyanalytic functions of order $n+1$, for suitable real coefficients $L_{j,n}$ (see \cite{Cacao,Falcao} and Section \ref{S2}). These polynomials are very special since they belong to the intersection of two different non-commutative function theories, namely the classical Fueter theory and the slice polyanalytic theory, moreover they have nice properties with respect to multiplication and derivation. Another important feature, see Theorem 3.10 in \cite{ads-fh}, is that any Fueter hyperholomorphic function $f$ of axial type admits a power series expansion in terms of the polynomials $P_n$ of the form $$f(x)=\displaystyle \sum_{n=0}^\infty P_n(x) u_n, \qquad u_n\in\mathbb H.$$ This fact allows to embed the space of Fueter hyperholomorphic functions of axial type, denoted by $\mathcal{AR}$, into a space consisting of series  of  slice polynalytic functions that we denote here by
$$\mathcal{SP}_\infty:=\mathcal{SP}_{1}+\mathcal{SP}_2+\cdots+\mathcal{SP}_{n+1}+\cdots,$$
where $\mathcal{SP}_{n}$ denotes the set of slice polyanalytic functions of order $n$.
More precisely we consider the subspaces of slice polyanalytic functions associated with the polynomials $(P_n)_{n\geq 0}$ defined by $$\mathcal{P}_n:=\left\lbrace P_n(x)\lambda, \textbf{  } \lambda\in\mathbb{H} \right\rbrace$$ \text{  and  } $$\mathcal{P}_\infty:=\bigoplus_{n=0}^\infty \mathcal{P}_n.$$
 Then, since the $P_n$ are the unique hyperholomorphic extensions of axial type of the real valued functions $(3x_0)^n$, it is possible to show that the space of hyperholomorphic functions of axial type $\mathcal{AR}$ corresponds to the space $\mathcal{P}_\infty$, i.e. $$\mathcal{AR}=\mathcal{P}_\infty.$$
 The previous subspaces of slice polyanalytic functions $\mathcal{P}_n$ were considered before from a different point of view and using a different terminology, namely they were called spaces of homogeneous special monogenic polynomials of degree $n$, see for example Lemma 1 in \cite{abul1990basic}. Using these ideas and identifications we show that it is always possible to embed this interesting subclass of special monogenic functions in a more general framework of slice polyanalytic functions. We use techniques from slice polyanalytic function theory to prove results on such special monogenic functions. In particular, in Proposition \ref{RFAM} we prove a Representation Formula in the monogenic setting using a slice polyanalytic approach.\smallskip

Furthermore, we note that these slice polyanalytic (and Fueter hyperholomorphic) polynomials $(P_n)_{n\geq 0}$ are just a particular case of a more general interesting construction which makes use of the classical Cauchy-Kovalevskaya extension theorem as we explain here.
Consider an entire real analytic and quaternionic-valued function $h$ of the real variables $x_1,x_2,x_3$.
The Cauchy-Kovalevskaya theorem guarantees the existence of a hyperholomorphic function $H$, its $CK$-extension.
We have, with $h_n=h^n$ and $H_n=H^{\odot n}$ (where $\odot$ denotes the Cauchy-Kovalevskaya product)
\begin{equation}
CK(h_n)\odot CK(h_m)=CK(h_{m+n})
\end{equation}
and so
\[
H_n=H_1^{n\odot}.
\]
We can see already here that obstructions occur; if we take quaternions $u$ and $v$, the $CK$-product $CK(h_nu)\odot CK(h_mv)$ will not be in general
be equal to $CK(h_{n+m}uv)$ since $h_m$ and $u$ do not commute. As a consequence, the $CK$-product will not be, in general, translated into convolution of
the coefficients of the expansions along the $H_n$. In spite of this,
with this new variable $H_1$ it is possible to define a number of counterparts of the classical reproducing kernel Hilbert spaces, with reproducing kernel of the form
\[
K(x,y)=\sum_{n\in I}\frac{H_n(x)\overline{H_n(y)}}{\alpha_n},\quad \alpha_n>0,\quad I\subset \mathbb N_0,
\]
converging in some neighborhood of the origin in $\mathbb R^4$. We already mention at this point
that the $CK$-product is not a law of composition for the Hardy space (defined below), and more generally, for series in the functions $H_n$.\smallskip

The choice $I=\mathbb N_0$ and $\alpha_n=1$ for every $n\in\mathbb N_0$ corresponds to the underlying Hardy space, consisting of functions
of the form
\[
f(x)=\sum_{n=0}^\infty H_n(x)f_n
\]
where $f_0,f_1,\ldots\in\mathbb H$ and satisfy $\sum_{n=0}^\infty |f_n|^2<\infty$. These functions are hyperholomorphic in
\[
\Omega=\left\{x\in\mathbb R^4\,;\, |H_1(x)|<1\right\}
\]
since the radius of convergence of Cauchy-Kovalevskaya product satisfies $\rho(H_n)\le (\rho(H_1))^n$; see \cite[Proposition 2.9, p. 131]{MR2124899}.\smallskip

A corresponding Schur analysis would consist in particular of the following problems:
\begin{itemize}
\item Characterize the contractive multipliers of this Hardy space. The definition has to be adapted to the present situation, where we lack
the convolution of the coefficients and the $CK$-product is not a law of composition.\smallskip

\item Study interpolation problems for these multipliers.\smallskip

\item Study the de Branges-Rovnyak spaces. These are families of Hilbert spaces of analytic functions, with reproducing kernels of various forms; see
\cite{dbhsaf1, dbbook,dbr1, dbr2,Dym_CBMS,Dmk,MR3497010,MR3617311}. Here we will focus on the counterpart of $\mathcal H(s)$ and $\mathcal L(\Phi)$ spaces,
whose reproducing kernel are of the form
  \begin{equation}
    \label{lphi}
\frac{1-s(z)\overline{s(w)}}{1-z\overline{w}}\quad {\rm and} \quad \frac{\Phi(z)+\overline{\Phi(w)}}{2(1-z\overline{w})}
\end{equation}
respectively.\smallskip

\end{itemize}

These various definitions and corresponding results need to be adapted to the present case, where we do not have a law of composition. We note that
the theory can be developed easily as in the classical way  when the coefficients are real, but this is of course restrictive. On the other hand, the theory using Fueter variables
works well because these variables are real when restricted to $x_0=0$ where $x_0$ denotes the real part of a quaternion.\mbox{}\\

There are important differences between the present treatment of Fueter hyperholomorphic functions and the treatment using Fueter variables; in the first case, the kernel functions are eigenvectors
of the backward shift, in the case of Fueter variables the kernel functions are eigenvectors of the three underlying backward-shift operator. Here the kernel functions are not eigenvectors of the backward-shift operator.\\
However, the present approach allows to make connections with the theory of slice polyanalytic functions, in particular with slice hyperholomorphic functions, and will also allow a simpler  functional calculus. Moreover, Toeplitz operators do appear in a natural way and play an important role.\\
In both cases, it is possible to develop a Schur type analysis. On the other hand, specific choices of the approach allow to make connections with  slice hyperholomorphic functions. We here consider the cases
\begin{equation}
\label{wx}
h(x)=  x_1{\mathbf e}_1+x_2{\mathbf e}_2+x_3{\mathbf e}_3\qquad{\rm and}\qquad w(x)=(1-h(x))(1+h(x))^{-1}
\end{equation}
and relate the underlying analysis with the Appell polynomials setting. Note that $w(0)=1\not=0$.\smallskip

A key fact used in the paper is that the hyperholomorphic functions considered are of axial type, and hence uniquely determined by their values on the real line.
\smallskip

We shall prove that the $CK$-extension of $x_1{\mathbf e}_1+x_2{\mathbf e}_2+x_3{\mathbf e}_3$ is
\begin{equation}
  CK(x_1{\mathbf e}_1+x_2{\mathbf e}_2+x_3{\mathbf e}_3)=x_1{\mathbf e}_1+x_2{\mathbf e}_2+x_3{\mathbf e}_3
  +3x_0=\zeta_1(x){\mathbf e}_1+\zeta_2(x){\mathbf e}_2+\zeta_3(x){\mathbf e}_3
\end{equation}
where $\zeta_i=x_i-x_0{\mathbf e}_i$, $i=1,2,3$ are the Fueter variables.
Moreover we have
\[
\begin{split}
CK((x_1{\mathbf e}_1+x_2{\mathbf e}_2+x_3{\mathbf e}_3)^{m})\odot CK((x_1{\mathbf e}_1+x_2{\mathbf e}_2+x_3{\mathbf e}_3)^{n})&=\\
&\hspace{-3cm}= CK((x_1{\mathbf e}_1+x_2{\mathbf e}_2+x_3{\mathbf e}_3)^{m+n}),
\end{split}
\]
so that we set
\[
CK((x_1{\mathbf e}_1+x_2{\mathbf e}_2+x_3{\mathbf e}_3)^{m})=\frac{Q_m(x)}{c_m}\stackrel{\rm def.}{=}P_m(x),
\]
where $Q_m$ denotes the $m$-th quaternionic Appell polynomial
(see \cite[(3.8)]{ads-fh} and \cite{DKS2019}). The  coefficients $c_m$ will be specified in Section \ref{S2}.
We are thus looking at a theory of hyperholomorphic functions of the variable
\begin{equation}
P_1(x)=\zeta_1(x){\mathbf e}_1+\zeta_2(x){\mathbf e}_2+\zeta_3(x){\mathbf e}_3=\frac{Q_1(x)}{c_1},
\end{equation}
equipped with the $CK$-product.
In our discussion it is crucial that
\begin{equation}
P_1^{n\odot}=P_n .
\end{equation}

We associate in a natural way to a Schur multiplier in the present setting a slice hyperholomorphic Schur multiplier; this allows to develop Schur analysis in the present setting.

\medskip

The de Branges-Rovnyak space associated with a Schur multiplier $S$ allows in the cases considered up to now to get a coisometric realization of the multiplier. In the complex setting, this is the celebrated
backward-shift realization (see \cite{Fuhrmann}). Here, the situation is a bit different. We can still associate to $S$ a coisometric operator matrix, in the form (in the current setting) of the
backward-shift realization, but the realization is on the level of the coefficients
(like in \cite{MR51:583} in the finite dimensional case).\\

We have given, or outlined, proofs of some classical results, for instance the extension result in Theorem \ref{dono} and the closely related Theorem \ref{toto123}. The reason is that the results
play a key role in this paper and some of the arguments are not necessarily well known in the Clifford analysis community. We apply them in the quaternionic setting in particular in Theorem \ref{ttt*}
and in Step 2 in the proof of Theorem \ref{toto}.\\

\medskip

The paper contains twelve sections, besides this Introduction. Section 2 contains some preliminary results. Section 3 contains results on reproducing kernel spaces and Toeplitz operators. In Section 4 we define the Hardy space in this framework, the backward-shift operator, Schur multipliers and their characterization. The Schur algorithm is presented in Section 5. Section 6 is focused on intrinsic functions, among which the polynomials $P_n$ and a description of  Fueter hyperholomorphic functions of axial type which are also intrinsic. In Section 7 we consider de Branges-Rovnyak spaces while in Section 8 we show how to define Blaschke functions, and the corresponding operator of multiplication which turns out to be an isometry. In Section 9 we consider the counterpart of Herglotz functions and multipliers and their associated kernels and Hilbert spaces. The next three sections concern the half-space case of Schur and Carath\'eodory multipliers. In Section 13 we summarize in a table a comparison between the various quaternionic settings.

\section{Preliminaries}\label{S2}
\setcounter{equation}{0}
This section contains three subsections: the first one introduces the map $\chi$; the second one introduces the Fueter variables and the polynomials obtained via the Appell
polynomials which will be the basis of our treatment. Finally, the third one shortly reviews positivity, analytic extensions and Toeplitz operators in the classical complex setting.
\subsection{Quaternions and the map $\chi$}
We will work in the skew field of quaternions, which is defined to be
$$\mathbb{H}=\lbrace{x=x_0+x_1\mathbf e_1+x_2\mathbf e_2+x_3\mathbf e_3\quad ; \ x_0,x_1,x_2,x_3\in\mathbb{R}}\rbrace$$ where the imaginary units satisfy the multiplication rules
$\mathbf e_i^2=-1$, $i=1,2,3$, $\mathbf e_1\mathbf e_2=-\mathbf e_2\mathbf e_1=\mathbf e_3$, $\mathbf e_2\mathbf e_3=-\mathbf e_3\mathbf e_2=\mathbf e_1$, $\mathbf e_3\mathbf e_1=-\mathbf e_1\mathbf e_3=
\mathbf e_2$.
The conjugate and the modulus of $x\in\mathbb{H}$ are defined by
$$\overline{x}={\rm Re}(x)-\underline{x} \quad \text{where} \quad {\rm Re}(x)=x_0, \quad \underline{x}=x_1\mathbf e_1+x_2\mathbf e_2+x_3\mathbf e_3$$
and $$\vert{x}\vert=\sqrt{x\overline{x}}=\sqrt{x_0^2+x_1^2+x_2^2+x_3^2},$$
respectively. The set of all  imaginary units is given by $\mathbb{S}=\lbrace{q\in{\mathbb{H}};q^2=-1}\rbrace.$ We note also that a domain $\Omega$ of $\mathbb{H}$ is called a slice domain if  $\Omega\cap{\mathbb{R}}$ is nonempty and for all $I\in{\mathbb{S}}$, the set $\Omega_I:=\Omega\cap{\mathbb{C}_I}$ is a domain of the complex plane $\mathbb{C}_I$.
If moreover, for every $x=u+Iv\in{\Omega}$, the whole sphere $$[q]:=\lbrace{u+Jv; \, J\in{\mathbb{S}}}\rbrace,$$
is contained in $\Omega$, we say that  $\Omega$ is an axially symmetric slice domain.

We can write a quaternion
as $x=z+w{\mathbf e_2}$ with $z=x_0+x_1\mathbf e_1$ and $w=x_2+x_3\mathbf e_1\in\mathbb C$. The map $\chi$ defined by
\[
  \chi(z+w{\mathbf e_2})=\begin{pmatrix}z&w\\
    -\overline{w}&\overline{z}\end{pmatrix}
  \]
  allows to transfer a number of problems from the quaternions to matrices in $\mathbb C^{2\times 2}$. We recall the following result, whose proof is immediate and will be omitted.

  \begin{lem}
    \label{rangechi}
    A matrix $M\in\mathbb C^{2\times 2}$ belongs to the range of $\chi$  if and only if it satisfies the symmetry
\begin{equation}
\label{symquater}
E^{-1}\overline{M}E=M,
\end{equation}
where $E=\begin{pmatrix}0&1\\-1&0\end{pmatrix}$.\\
  \end{lem}

For matrices and operators, there are various ways to define $\chi$. Let
$X=A+B\mathbf e_2\in\mathbb H^{r\times s}$. We set
\[
  \chi(X)=\begin{pmatrix}A&B\\
    -\overline{B}&\overline{A}\end{pmatrix}.
\]
We define for a block matrix $(X_{jk})$ with $X_{jk}\in\mathbb H^{r\times s}$,
\[
(\chi(X))_{jk}=\chi(X_{jk}).
\]
For matrices $M_1$, $M_2$ possibly infinite, with block entries in $\mathbb{H}^{r\times s}$ and $\mathbb{H}^{s\times t}$, respectively, we have the property
\begin{equation}
\chi(M_1M_2)=\chi(M_1)\chi(M_2).
\end{equation}
We note that $\chi$ will not be compatible with the $CK$-product. An important tool in the paper consists of bounded block Toeplitz operators, with blocks in the range of $\chi$:
\[
\label{T-op}
T=\left(\chi(X_{j-k})\right)_{j,k=0}^\infty.
\]
and we will need the following result, set for general operators.

\begin{prop}
\label{2-9-0}
The operator $\tau$
\begin{equation}
  \tau=\left(X_{jk}\right)_{j,k=0}^\infty
\end{equation}
is bounded from $\ell_2(\mathbb N_0,\mathbb H^s)$  into $\ell_2(\mathbb N_0,\mathbb H^r)$ if and only if the operator $T$ defined by
\begin{equation}
\label{T-op-1}
T=\left(\chi(X_{jk})\right)_{j,k=0}^\infty ,
\end{equation}
where the $X_{jk}$ are matrices in  $\mathbb H^{r\times s}$ is bounded
from $\ell_2(\mathbb N_0,\mathbb C^{2s})$  into $\ell_2(\mathbb N_0,\mathbb C^{2r})$, and both operators have same norm.
\end{prop}

\begin{proof}
  One direction is clear: If $\tau$ is bounded, there is a constant $K>0$ such that
  \[
\sum_{i,j,k=0}^{\infty} q_i^*X_{ij}X_{jk}^*q_k\le K\sum_{i=0}^\infty q_i^*q_i, \quad \text{for any } q_i\in\mathbb{H}^r.
\]
Applying $\chi$ we get
\begin{equation}
  \label{12214343}
\sum_{i,j,k=0}^\infty \chi(q_i)^*\chi(X_{ij})\chi(X_{jk})^*\chi(q_k)\le K\chi\left(\sum_{i=0}^\infty  q_i^*q_i I_2\right),
\end{equation}
where $I_2$ denotes the identity matrix of order $2$.
Multiplying this inequality by $\begin{pmatrix}1&0\end{pmatrix}$ on the left and by its transpose on the right, we get the result.\smallskip

Conversely, if $T$ is bounded there exists $K>0$ such that for all $u_1,v_1,\ldots\in\mathbb C^r$,
\begin{equation}
\label{qwerty123}
\sum_{i,j,k=0}^\infty  \begin{pmatrix}u_i^*&v_i^*\end{pmatrix}\chi(X_{ij})\chi(X_{jk})^*\begin{pmatrix}u_k\\ v_k\end{pmatrix}\le  K\left(\sum_{i=0}^\infty u_i^*u_i+v_i^*v_i\right)
\end{equation}
and
\begin{equation}
\label{qwerty1234}
\sum_{i,j,k=0}^\infty   \begin{pmatrix}-v_i^t&u_i^t\end{pmatrix}\chi(X_{ij})\chi(X_{jk})^*\begin{pmatrix}-\overline{v_k}\\ \overline{u_k}\end{pmatrix}\le  K\left(\sum_{i=0}^\infty u_i^*u_i+v_i^*v_i\right).
\end{equation}
Set
\[
  \begin{split}
e^*&=\begin{pmatrix} u_1^*&v_1^*&u_2^*&v_2^*&\cdots\end{pmatrix}\\
f^*&=\begin{pmatrix} -v_1^t&u_1^t&-v_2^t&u_2^t&\cdots\end{pmatrix}
  \end{split}
\]
and denote by $M$ the bounded linear operator with $i,k$ block equal to $M_{ik}=\chi(Y_{ik})$, with
$Y_{ik}=\sum_{j=1}^\infty X_{ij}X_{jk}^*$, $i,j=1,\ldots$. We have
\begin{equation}
  \label{lkjhg}
\begin{pmatrix}
  e^*\\f^*
\end{pmatrix}M\begin{pmatrix}e &f\end{pmatrix}
=\begin{pmatrix}e^*Me&e^*Mf\\
  f^*Me&f^*Mf\end{pmatrix}\le \|M\|(e^*e+f^*f) I_2.
\end{equation}
 On the other hand, setting
\[
q_i^*=u_i^*+v_i^*{\mathbf e_2},\quad i=1,2,\ldots
\]

we rewrite \eqref{qwerty123}-\eqref{qwerty1234} as
\[
 \begin{pmatrix}
  e^*\\f^*
\end{pmatrix}M\begin{pmatrix}e &f\end{pmatrix}= \begin{pmatrix}\chi(q_1)^*&\chi(q_2)^*&\cdots\end{pmatrix}\begin{pmatrix}\chi(Y_{11})&\chi(Y_{12})&\cdots&\\
      \chi(Y_{12})^*&\chi(Y_{22})&\cdots\\
     \vdots & &\ddots &\\
      & & &\\
      \end{pmatrix}
  \begin{pmatrix}\chi(q_1)\\ \chi(q_2)\\ \vdots\end{pmatrix}.
\]
Comparing with \eqref{lkjhg} we get
\[
\begin{pmatrix}\chi(q_1)^*&\chi(q_2)^*&\cdots\end{pmatrix}\begin{pmatrix}\chi(Y_{11})&\chi(Y_{12})&\cdots&\\
      \chi(Y_{12})^*&\chi(Y_{22})&\cdots\\
     \vdots & &\ddots &\\
      & & &\\
      \end{pmatrix}
  \begin{pmatrix}\chi(q_1)\\ \chi(q_2)\\ \vdots\end{pmatrix}\le\|M\|\chi(e^*e+f^*f)
\]
and hence the result, since $\chi$ preserves order.\smallskip

The claim on the norms being the same follows from the previous inequality and \eqref{12214343}.
\end{proof}

\subsection{Various notions of hyperholomorphy and homogeneous polynomials}
In this section we briefly review the setting of Fueter variables and the Cauchy-Kovalevskaya product.
We recall that left-hy\-per\-ho\-lo\-mor\-phic functions (we will usually just say {\sl hyperholomorphic} in the sequel) are solutions of the equation $Df=0$, where $D$ denotes the Cauchy-Fueter operator
\begin{equation}
D=\frac{\partial}{\partial x_0}+\mathbf e_1\frac{\partial}{\partial x_1}+\mathbf e_2\frac{\partial}{\partial x_2}+\mathbf e_3\frac{\partial}{\partial x_3}.
\end{equation}
These functions are widely studied in the literature. They are, in particular, harmonic functions in four real variables. Unfortunately, the monomials $x^n$ in the quaternionic variable $x$ are not in the kernel of the Cauchy-Fueter operator, not even when $n=1$. However, hyperholomorphic functions admit a series expansion in terms of the so-called Fueter variables, as we shall see below. We point out that, in this paper, we shall provide only the notions and results needed in the sequel, and for further information on this class of functions we refer the reader to \cite{MR2089988, gurlebeck2008application}.\smallskip

In view of the Cauchy-Kovalevskaya theorem, a linear system of first order differential equations satisfied by the real components of $f$ has a unique solution when the function
$\varphi(x_1,x_2,x_3)=f(0,x_1,x_2,x_3)$ is pre-assigned (and assumed real analytic). The function $f$ with this initial condition and solution of $Df=0$ is called the Cauchy-Kovalevskaya extension of
$\varphi$ (here written as $CK(\varphi)$ and abbreviated as $CK$-extension).

Among the important solutions of the equation $Df=0$ there are the Fueter variables
\begin{equation}
\zeta_j(x)=x_j-\mathbf e_jx_0,\quad j=1,2,3,
\end{equation}
corresponding respectively to $\varphi_j(x)=x_j$, $j=1,2,3$, and their symmetric products
\begin{equation}
\label{zetanu}
\zeta^\nu=\zeta_1^{\nu_1\times}\times \zeta_2^{\nu_2\times}\times \zeta_3^{\nu_3\times},\quad \nu=(\nu_1,\nu_2,\nu_3)\in\mathbb N_0^3,
\end{equation}
with, for $a_1,\ldots, a_n\in\mathbb H$
\[
a_1\times   a_2\times\cdots\times a_n=\frac{1}{n!}\sum_{\sigma\in S_n}a_{\sigma(1)}a_{\sigma(2)}\cdots a_{\sigma(n)}.
\]
We note that $\zeta^\nu=CK(x^\nu)$ with $x^\nu=x_1^{\nu_1}x_2^{\nu_2}x_3^{\nu_3}$. A direct proof that $\zeta^\nu$ given by \eqref{zetanu} is hyperholomorphic (and hence is the $CK$-extension of $x^\nu$)
is not trivial and can be found in
\cite[\S 3]{MR0265618}, \cite{MR1509533}. The argument works also in the split quaternion setting.
See \cite[p. 333-334]{MR3819695}. It is important to note that every function hyperholomorphic in a neighborhood of the origin can be written as a convergent power series
in the form of a Fueter series
\begin{equation}
\label{fueterseries}
f(x)=\sum_{\nu\in\mathbb N_0^3} \zeta^\nu f_\nu
\end{equation}
where the coefficients $f_\nu$ belong to $\mathbb H$. See \cite{bds}. A proof based on the Gleason problem can be found in \cite{MR2124899}.\smallskip

Using the CK-extension one can define a product that preserves the hyperholomorphicity, the so-called CK-product denoted by $\odot$. The idea to compute the CK-product is the following: if $f$ and $g$ are two hyperholomorphic functions, we take their restriction to $x_0=0$, which are real analytic functions, and consider their pointwise multiplication. Then, we take the Cauchy-Kowalevskaya extension of this pointwise product, which exists and is unique, to define
\begin{equation}\label{oCK}
f\odot g=CK(f(0,x_1,x_2,x_3)g(0,x_1,x_2,x_3)),
\end{equation}
 see \cite{gurlebeck2008application}. Moreover, we note that the following formula holds $$CK[\varphi(\underline{x}\,)](x)=\exp\left(-x_0\partial_{\underline{x}\,}\right)
[\varphi(\underline{x}\,)](x).$$

For power series of the form \eqref{fueterseries} the $CK$-product is a convolution on the coefficients along the basis $\zeta^\nu$,
and in particular
\begin{equation}
\label{pqzeta}
\zeta^{\nu}p\odot\zeta^{\mu}q=\zeta^{\nu+\mu}pq,\quad p,q\in\mathbb H,\quad \mu,\nu\in\mathbb N_0^3,
\end{equation}
where $\nu+\mu$ is defined componentwise, see \cite{bds,MR618518}.\smallskip

We now turn to a bound for the $CK$-product; see also \cite[pp. 132-133]{MR2124899}.

\begin{lem}
\label{fog}
Let $\rho>0$. There exists $\epsilon>0$ such that:
\[
x_0^2+x_j^2<\epsilon,\quad j=1,2,3,\quad\Longrightarrow\quad \sum_{\substack{\alpha\in\mathbb N_0^3\\ \alpha\not=(0,0,0)}}|\zeta(x)^\alpha | |f_\alpha|<\rho .
\]
Then
\begin{equation}
\label{ineq23}
\left|\left(\sum_{\substack{\alpha\in\mathbb N_0^3\\ \alpha\not=(0,0,0)}}\zeta(x)^\alpha f_\alpha\right)^{\odot n}\right|<\rho^n,\quad n=1,2,3,\ldots
\end{equation}
\end{lem}

\begin{proof}
We first note that for $\nu\in\mathbb N_0^3$
\begin{equation}
\label{ineq1}
|\zeta^\nu(x)|\le \epsilon^{|\nu|}, \quad{\rm where}\quad x_0^2+x_j^2<\epsilon,\quad j=1,2,3.
\end{equation}
The existence of $\epsilon$ follows from the dominated convergence theorem, and the first assertion follows. Then, setting $g(x)=\sum_{\beta\in\mathbb N_0^3}\zeta(x)^\beta g_\beta$ we have
\[
\begin{split}
|(f\odot g)(x)|&\le |\sum_{\gamma\in\mathbb N_0^3}|\zeta^\gamma(x)|\cdot|\sum_{\substack{\alpha,\beta\in\mathbb N_0^3\\\alpha+\beta=\gamma}}f_\alpha g_\beta|\\
&\le\sum_{\alpha,\beta\in\mathbb N_0^3}\epsilon^{|\alpha|+|\beta|}|f_\alpha|\cdot|g_\beta|\\
&\le\left(\sum_{\alpha\in\mathbb N_0^3}\epsilon^{|\alpha|}|f_\alpha|\right)\left(\sum_{\beta\in\mathbb N_0^3}\epsilon^{|\beta|}|b_\beta|\right),
\end{split}
\]
from which \eqref{ineq23} follows.
\end{proof}
\begin{lem}
The $CK$-extension of $x_1{\mathbf e}_1+x_2{\mathbf e}_2+x_3{\mathbf e}_3$ to a Fueter hyperholomorphic function is
\begin{equation}\label{216}
  x_1{\mathbf e}_1+x_2{\mathbf e}_2+x_3{\mathbf e}_3
  +3x_0=\zeta_1(x){\mathbf e}_1+\zeta_2(x){\mathbf e}_2+\zeta_3(x){\mathbf e}_3
\end{equation}
where $\zeta_1,\zeta_2,\zeta_3$ are the Fueter variables.
\end{lem}

\begin{proof}
It suffices to note that the function
\[
\begin{split}
  \zeta_1(x){\mathbf e}_1+\zeta_2(x){\mathbf e}_2+\zeta_3(x){\mathbf e}_3&=(x_1-{\mathbf e}_1x_0){\mathbf e}_1+(x_2-{\mathbf e}_2x_0){\mathbf e}_2+(x_3-{\mathbf e}_3x_0){\mathbf e}_3\\
&  =x_1{\mathbf e}_1+x_2{\mathbf e}_2+x_3{\mathbf e}_3+3x_0
\end{split}
\]
is Fueter hyperholomorphic and its restriction to $x_0=0$ is the given function.
\end{proof}
Let us now introduce another type of Fueter hyperholomorphic homogeneous polynomials, see \cite{Laguerre, DKS2019, FalcaoC}:
\begin{defn}
The polynomials
\begin{equation}\label{Qem}
Q_m(x)=\sum_{j=0}^m T^m_j x^{m-j}\bar{x}^j
\end{equation}
 where
\[
T^m_j=\frac{2(m-j+1)}{(m+1)(m+2)},\quad m=0,1,\ldots ,
\]
are called the $m$-th quaternionic Appell polynomials.
\end{defn}

The polynomials $(Q_m)_{m\geq 0}$  are Fueter regular. Moreover, a generalized Fueter regular exponential function associated to these polynomials  was considered in the literature, see for example \cite{Laguerre}.
Another interesting feature of the quaternionic Appell polynomials is that they can be obtained by applying the Fueter mapping applied to the standard quaternionic monomials $x^m$. In particular, in \cite{DKS2019} the following formula is proved \begin{equation}
\displaystyle Q_m(x)=-\frac{\Delta(x^{m+2})}{2(m+1)(m+2)}, \quad m=0,1,...
\end{equation}
 We then define another kind of Fueter hyperholomorphic polynomials by
 \[
P_m(x)\stackrel{\rm def.}{=}\frac{Q_m(x)}{c_m},
\]
where
\[
c_m=\sum_{j=0}^m(-1)^jT^m_j.
\]

We have the following relation between the polynomials $P_m$ and the $CK$-extension of $x_1{\mathbf e}_1+x_2{\mathbf e}_2+x_3{\mathbf e}_3$ in \eqref{216}:
\begin{prop}
The following equality holds:
  \begin{equation}\label{CKPm}
CK((x_1{\mathbf e}_1+x_2{\mathbf e}_2+x_3{\mathbf e}_3)^{m})=P_m(x).
\end{equation}
\end{prop}
\begin{proof}
The proof is simple and it is based on the fact that at both hand sides there are monogenic functions which coincide on $x_0=0$:
\[
(x_1{\mathbf e}_1+x_2{\mathbf e}_2+x_3{\mathbf e}_3)^{m}=\dfrac{1}{c_m}\sum_{j=0}^m T^m_j \underline{x}^{m-j}(-\underline{x})^j= \dfrac{1}{c_m}\sum_{j=0}^m (-1)^j T^m_j \underline{x}^m=  \underline{x}^m.
\]
\end{proof}
\begin{rem}
The polynomials $Q_m$ are
called Appell since they satisfy the Appell property
$$
\frac 12 \overline{D} Q_m =m Q_{m-1}, \qquad m\geq 1 ;
$$
the $P_m$ do not respect such a property, since $$\frac 12 \overline{D} P_m =m \frac{c_{m-1}}{c_m} P_{m-1}, \qquad m\geq 1,$$ however, they behave  better  with respect to the $CK$-product, as we shall see below. In particular, for even indexes of the form $m=2k$, the Appell property is still satisfied by the polynomials $(P_{2k})_{k\geq 0}$ since we have $c_{m-1}=c_m$ in this case.
\end{rem}
In what follows, we are looking at a theory of hyperholomorphic functions of the variable
\begin{equation}
P_1(x)=\frac{Q_1(x)}{c_1}=\zeta_1(x){\mathbf e}_1+\zeta_2(x){\mathbf e}_2+\zeta_3(x){\mathbf e}_3,
\end{equation}
with the $CK$-product.
Moreover, note that
\begin{equation}
P_1(x_0)=3x_0.
\end{equation}

The $\odot$-product is not a convolution on the coefficients of the $P_n$: in opposition to \eqref{pqzeta} we have, in general,
\begin{equation}
P_np\odot P_mq\not =P_{n+m}pq,\quad n,m\in\mathbb N,\quad p,q\in\mathbb H.
\end{equation}
In particular, in general
\begin{equation}
  \label{obstruct}
(1-P_1q)^{-\odot}\not =\sum_{n=0}^\infty P_nq^n
\end{equation}
for $q\in\mathbb H$ in a neighborhood of the origin.\smallskip

This obstruction is the source of the main difficulties and new results in the present paper.
Still, we have the following simple result, which plays a key role in the computations (see in particular
\eqref{facto-pn}).

\begin{lem}\label{PnPm}
It holds that
\begin{equation}
(P_n\odot P_m)(x)=P_{n+m}(x),
\end{equation}
and, in particular,
\begin{equation}\label{P1n}
P_n(x)=(P_1(x))^{\odot n},\quad n=1,2,\ldots
\end{equation}
Furthermore, for $n,m,k\in\mathbb N_0$ and $u\in\mathbb H^r$
\begin{equation}
\label{nmk}
(P_n\odot(P_m\odot P_ku))=P_{n+m+k}u=P_{n+m}\odot P_ku.
\end{equation}
\end{lem}

\begin{proof}
We have
\[
\begin{split}
CK((x_1{\mathbf e}_1+x_2{\mathbf e}_2+x_3{\mathbf e}_3)^{n})\odot CK((x_1{\mathbf e}_1+x_2{\mathbf e}_2+x_3{\mathbf e}_3)^{m})&=\\
&\hspace{-3cm}= CK((x_1{\mathbf e}_1+x_2{\mathbf e}_2+x_3{\mathbf e}_3)^{n+m})\\
&\hspace{-3cm}=\left(\zeta_1(x){\mathbf e}_1+\zeta_2(x){\mathbf e}_2+\zeta_3(x){\mathbf e}_3\right)^{\odot(n+m)},
\end{split}
\]
where we used \eqref{CKPm} in the last equality. In particular, by iteration, we obtain \eqref{P1n}.
The last claim follows from restricting the equalities for $x_0=0$, and checking that they are equal to
$(x_1{\mathbf e}_1+x_2{\mathbf e}_2+x_3{\mathbf e}_3)^{n+m+k}u$.
\end{proof}

\begin{cor}
\label{coro234}
For every $\rho>0$ there exists $\epsilon>0$ such that
\[
x_0^2+x_j^2<\epsilon,\quad j=1,2,3,\quad\Longrightarrow\quad |P_n(x)|<\rho^n.
\]
\end{cor}

\begin{proof}
This follows by induction from Lemma \ref{fog} with $f=P_1$ and $g=P_{n}$, $n\in\mathbb N$.
\end{proof}

We have (see e.g. \cite[(2.19) p. 135]{MR2124899} with $A_u={\mathbf e}_u$, $u=1,2,3$)
\begin{equation}
\label{ze1ze2ze3}
(\zeta_1(x){\mathbf e}_1+\zeta_2(x){\mathbf e}_2+\zeta_3(x){\mathbf e}_3)^{m\odot}=\sum_{|\nu|=m}\zeta^{\nu}{\mathbf e}^{\nu}\frac{|\nu|!}{\nu!}
\end{equation}
where $\zeta^\nu$ is defined in \eqref{zetanu}.

\begin{rem}
{\rm One could take the $CK$-extension of another linear combination such as $t_1x_1{\mathbf e}_1+t_2x_2{\mathbf e}_2+t_3x_3{\mathbf e_3}$, namely
\begin{equation}
t_1\zeta_1{\mathbf e_1}+t_2\zeta_2{\mathbf e}_2+t_3\zeta_3{\mathbf e}_3=t_1x_1{\mathbf e}_1+t_2x_2{\mathbf e}_2+t_3x_3{\mathbf e_3}+(t_1+t_2+t_3)x_0
\end{equation}
and develop a similar theory.}
\end{rem}

Let $f(x_0)=\sum_{n=0}^\infty x_0^na_n$ (with $a_0,a_1,\ldots\in\mathbb H$) be a real analytic function near the origin. It does not have a unique hyperholomorphic extension of course, as seen by taking
\[
\zeta_1(x){\mathbf e}_1 \quad   {\rm and}\quad \zeta_2(x){\mathbf e}_2,
\]
in fact both functions are equal to $x_0$ on the real line. However the extension becomes unique by requiring that it is of a special form:

\begin{lem}
Let $f(x_0)=\sum_{n=0}^\infty x_0^na_n$, $a_n\in\mathbb{H}$ be a real analytic function near the origin. It has a unique (left) hyperholomorphic extension of the form $f(x)=\sum_{n=0}^\infty P_n(x)b_n$,
namely
\begin{equation}\label{specialmon}
f(x)=\sum_{n=0}^\infty P_n(x)\frac{a_n}{3^n}.
\end{equation}
Similarly, its unique right hyperholomorphic extension is
\[
g(x)=\sum_{n=0}^\infty \frac{a_n}{3^n}P_n(x).
\]
\end{lem}

\begin{proof}
The function $f(x)$ is indeed an extension of the required form. If there is another one, say $\tilde{f}(x)=\sum_{n=0}^\infty P_n(x)d_n$ we get when setting $x_1=x_2=x_3=0$
\[
\sum_{n=0}^\infty (3x_0)^nb_n=  \sum_{n=0}^\infty (3x_0)^nd_n
\]
and so $b_n=d_n$, $n=0,1,\ldots$. A similar reasoning works for $g$.
\end{proof}
\begin{rem}{\rm We note that the polynomials $P_m$ and $Q_m$ are both left and right hyperholomorphic and in fact $P_m(x)$ corresponds to both the left and right CK-extension of $(x_1{\mathbf e}_1+x_2{\mathbf e}_2+x_3{\mathbf e}_3)^{m}$.}
\end{rem}
Actually, the previous result is a particular case of a more general result that holds for Fueter hyperholomorphic functions of axial type, whose definition which comes from the more general case of axially monogenic functions, see \cite{DeSS}, is the following:
 \begin{defn}
A Fueter hyperholomorphic function is of axial type (or axially hyperholomorphic) if it is of the form
\[
A(x_0,|\underline{x}|) + \dfrac{\underline{x}}{|\underline{x}|} B(x_0,|\underline{x}|),
\]
where $A$, $B$ are quaternionic valued.
 \end{defn}
The condition that a function  $f$ of axial type is in the kernel of the Cauchy-Fueter operator $D$ translates into the Vekua system
\[
\partial_{x_0}A- \partial_{\rho} B=\frac{2}{\rho} B,\qquad \partial_{x_0}B+\partial_{\rho}A=0, \qquad \rho=|\underline{x}|.
\]
Starting from any real analytic function $A(x_0)$ it is possible to construct its unique Fueter hyperholomorphic extension of axial type.\smallskip

We will say that a matrix-valued hyperholomorphic function  is of {\em axial type} if all its entries, as matrix, are of axial type.

\begin{rem}\label{rmk212}
Functions of the form \eqref{specialmon} are quaternionic special monogenic according to the terminology in \cite{abul1990basic}. Any quaternionic special monogenic function in a neighborhood of the origin is of axial type. In fact any polynomial $P_m(x)$ is the sum of terms of the form
\[
(x_0^2+|\underline{x}|^2)^k (x_0\pm \underline{x})^h, \qquad 2k+h=m,\, k\geq 0, \, h\geq 0
\]
which are evidently of axial type. This fact was already noted in \cite{Laguerre}, Property 2. A Fueter regular function represented by a uniformly convergent series of the form \eqref{specialmon}
is such that
\[
\begin{split}
\sum_{m\geq 0} A_m(x_0,|\underline{x}|)+\dfrac{\underline{x}}{|\underline{x}|} B_m(x_0,|\underline{x}|)&=\sum_{m\geq 0} A_m(x_0,|\underline{x}|)+\dfrac{\underline{x}}{|\underline{x}|} \sum_{m\geq 0}B_m(x_0,|\underline{x}|)\\
&= A(x_0,|\underline{x}|)+\dfrac{\underline{x}}{|\underline{x}|} B(x_0,|\underline{x}|)
\end{split}\]
where the pair $A$, $B$ satisfy the Vekua system. Conversely, any function of axial type is of the form \eqref{specialmon}, by Theorem 3.10 in \cite{ads-fh}.
\end{rem}
We recall the notion of slice polyanalytic functions, see \cite{ADSP2019}.
\begin{defn}
A real differentiable function $f:\Omega\longrightarrow \mathbb H$  of the form $$f(x)=\alpha(u,v)+I\beta(u,v), \quad x=u+Iv\in \Omega$$ with $\alpha(u,-v)=\alpha(u,v)$ and $\beta(u,-v)=-\beta(u,v)$  is called left slice polyanalytic of order $N$, if for all $I\in\mathbb{S}$, $f_I$ is left polyanalytic of order
$N$ on $\Omega_I$, namely if  $$ \overline{\partial_I}^N f(u+Iv):=
\frac{1}{2^N}\left(\frac{\partial }{\partial u}+I\frac{\partial }{\partial v}\right)^Nf_I(u+Iv)=0.
$$
\end{defn}
When $N=1$ the notion coincides with that one of slice hyperholomorphicity (slice regularity).

We have the following characterization, see \cite[Proposition 3.6]{ADSP2019}:
\begin{prop}\label{Poly-Dec}
A function $f$ defined on a domain $\Omega\subseteq\mathbb H$ is slice polyanalytic of order $N$ on $\Omega$ if and only if it can be written as
\begin{equation}\label{sum}
f(x):=\displaystyle\sum_{k=0}^{N-1}\overline{x}^kf_k(x)
\end{equation}
where $f_0,...,f_{N-1}$ are slice regular functions in $\Omega$.
\end{prop}
As a consequence:
\begin{cor}\label{PmPol}
The polynomial $P_m$ is slice polyanalytic of order $m+1$.
\end{cor}
\begin{proof}
For all $0 \leq k\leq m$, we set $f_k(x)=\dfrac{T^m_k}{c_m}x^{m-k}$. It is clear that all $f_k$ are slice regular functions on $\Omega$, being polynomials in the variable $x$. Moreover, we note that $$P_m(x)=\displaystyle \sum_{k=0}^m\overline{x}^kf_k(x), \forall x\in \Omega.$$
Hence, the thesis follows using Proposition \ref{Poly-Dec}.
\end{proof}
In the definition of the polynomials $P_m$ we note that to write the monomials as $x^j\bar x^\ell$ or $\bar x^\ell x^j$ does not make any difference since $x\bar x=\bar x x$.\\
A Representation Formula  for Fueter hyperholomorphic functions of axial type is immediately deduce from the fact that they are slice functions, see \cite{GP}, so we have:
\begin{prop}\label{RFAM}
Let $f:\, \Omega\subset\mathbb{H}\to\mathbb H$ be a Fueter hyperholomorphic function of axial type  where $\Omega$ is an axially symmetric slice domain. Let $J\in\mathbb{S}$, then for any $x=u+I_xv\in\Omega$ the following equality holds :
$$
\displaystyle f(u+I_xv)= \frac{1}{2}\left[f_J(u+Jv)+f_J(u-Jv)\right]+\frac{I_xJ}{2}\left[f_J(u-Jv)-f_J(u+Jv)\right].
$$
\end{prop}
\begin{proof}
We note that the Fueter hyperholomorphic polynomials $(P_m)_{m\geq 0}$ are slice polyanalytic of order $m+1$ thanks to Corollary \ref{PmPol}. Thus, the Representation Formula is an immediate consequence.
\end{proof}

\begin{rem}
An alternative proof of the previous Representation Formula in the Fueter hyperholomorphic context consists to apply Proposition 3.13 in \cite{ADSP2019} to each polynomial $P_m$.
\end{rem}
We conclude this part with a result which will be used in the sequel while dealing with kernels:
\begin{prop}\label{anti}
The polynomial $\overline{P_m(x)}$ is right anti-hyperholomorphic in $x$, namely it satisfies
\[
\overline{P_m(x)} \overline{D}=
\frac{\partial \overline{P_m}}{\partial x_0} -\frac{\partial \overline{P_m}}{\partial x_1} \mathbf e_1-\frac{\partial \overline{P_m}}{\partial x_2}\mathbf e_2-\frac{\partial \overline{P_m}}{\partial x_3}\mathbf e_3=0.\]
More in general, if $f$, $g$ are left hyperholomorphic, $\overline{f \odot g}=\overline{g}\odot_R\overline{f}$ is right anti-hyperholomorphic.
\end{prop}
\begin{proof}
We immediately have:
\[
\overline{P_m(x)} \overline{D}= \overline{D P_m(x)} =0,
\]
and the first assertion follows. Then we have
$$
\overline{f \odot g}=\overline{CK(f_{|x_0=0} \odot g _{|x_0=0}})= {CK(\overline{f_{|x_0=0} \odot g _{|x_0=0}}}).
$$
We now note that
$$
CK(\overline{f_{|x_0=0} \odot g _{|x_0=0}})= CK(\overline{g_{|x_0=0}} \odot_R \overline{f _{|x_0=0}})=CK(\overline{g_{|x_0=0}}) \odot_R CK(\overline{f _{|x_0=0}})=\overline{g}\odot_R\overline{f}
$$
which concludes the proof.
\end{proof}

\subsection{Positivity, analytic extension and Toeplitz operators}
This section considers the complex variable setting.
Recall first that a $\mathbb C^{n\times n}$-valued function $K(z,w)$ defined for $z,w$ varying in some set $\Omega$ is called positive definite if for every choice of
$N\in\mathbb N_0$ and $w_1,\ldots,w_N\in\Omega$ the block matrix $(K(w_j,w_k))_{j,k=1}^N$ is non-negative. Associated to $K(z,w)$ is a uniquely defined Hilbert space of
$\mathbb C^n$-valued functions defined on $\Omega$, denoted here $\mathcal H(K)$, and with the properties:\\
$(1)$ For every $c\in\mathbb C^n$ and $w\in\Omega$, the function $K_wc\,:\, z\mapsto \, K(z,w)c$ belongs to $\mathcal H(K)$, and\\
$(2)$ For every $f\in\mathcal H(K)$ and $w,c$ as above,
\begin{equation}
\langle f,K_wc\rangle=c^*f(w).
\end{equation}
$\mathcal H(K)$ is called the reproducing kernel Hilbert space with reproducing kernel $K(z,w)$, and there is a one-to-one correspondence between reproducing kernel
Hilbert spaces and positive definite functions; see \cite{aron,meschkowski,saitoh}.
We recall the following result, which originates with the work of Donoghue \cite{donoghue}. We take a real neighborhood of the origin, but it could be replaced by any other zero set in
the open unit disk.

\begin{thm}
\label{dono}
Let $s$ be a ${\mathbb C}^{r\times t}$-valued function defined in a neighborhood $(-\epsilon,\epsilon)$ of the origin, and such that the kernel
\begin{equation}
\label{KSx0y0}
\frac{I_r-s(a)s(b)^*}{1-ab}
\end{equation}
is positive definite in $(-\epsilon,\epsilon)$.  Then $s$ has a (uniquely defined) analytic and contractive extension to the open unit disk.
\end{thm}

\begin{proof} The proof can be found in e.g. \cite[pp. 45-46]{MR1638044}. For completeness we outline it. We set $r=s=1$ to simplify the notation. Let $\rho_w(z)=1-z\overline{w}$.
The function $1/\rho_w(z)$ is positive definite in the open unit disk $\mathbb D$, with reproducing kernel Hilbert space the Hardy space of the open unit disk, denoted
$\mathbf H_2(\mathbb D)$, and
consisting of the power series $f(z)=\sum_{n=0}^\infty a_nz^n$ with complex coefficients satisfying $\|f\|_2^2\stackrel{\rm def.}{=}\sum_{n=0}^\infty |a_n|^2<\infty$. The formula
\[
T(1/\rho_x)=\frac{\overline{s(x)}}{\rho_x},\quad x\in(-1/3,1/3)
\]
extends linearly to a densely defined operator $T$ from $\mathbf H_2(\mathbb D)$ into itself. The positivity of the kernel \eqref{KSx0y0} and the definition of the inner
product in the Hardy space implies that $T$ is a contraction, and hence extends to an everywhere defined contraction, still denoted by $T$, from $\mathbf H_2(\mathbb D)$ into itself.
Let $f\in\mathbf{H}_2(\mathbb D)$. The adjoint of $T$ satisfies:
\begin{equation}
\label{tt*}
(T^*f)(x)=\langle T^*f,\frac{1}{\rho_x}\rangle=\langle f,\frac{\overline{s(x)}}{\rho_x}\rangle=s(x)f(x), \quad x\in(-1/3,1/3).
\end{equation}
Take first $f(z)=1$.
Since $T^*1$ is analytic in the open unit disk, it is an analytic extension of $s(x)$, $x\in(-1/3,1/3)$.
It remains to check that $T^*1$ is contractive in $\mathbb D$. Equation \eqref{tt*} extends analytically to
\begin{equation}
\label{toep}
(T^*f)(z)=(T^*1)(z)f(z),\quad z\in\mathbb D.
\end{equation}
So $T^*$ is the operator of multiplication by $T^*1$. Since it is bounded, the formula for the adjoint of a multiplication operator acting in a reproducing kernel Hilbert space gives
\[
T^*(1/\rho_w)  =\frac{\overline{T^*1(w)}}{\rho_w}.
\]
Since it is contractive, writing that $\|T^*1/\rho_w\|\le \|1/\rho_w\|$ we get
\[
\frac{|T^*1(w)|^2}{1-|w|^2}\le\frac{1}{1-|w|^2},\quad w\in\mathbb D,
\]
and hence $T^*1$ takes contractive values in the open unit disk.
\end{proof}

We also recall (we refer to \cite{nik1} for more information on Toeplitz operators):
\begin{thm}
Let $s$ be a $\mathbb C^{r\times t}$-valued function analytic in the open unit disk, with power series expansion $s(z)=\sum_{n=0}^\infty s_nz^n$. Then, $s$ is
contractive in the open unit disk if and only if the lower triangular block-Toeplitz operator
\begin{equation}\label{Tesse}
T_s=\begin{pmatrix} s_0&0&0&\cdots\\
  s_1&s_0&0&\cdots\\
  s_2&s_1 &s_0 &0&\\
 \vdots & \ddots &\ddots &\ddots \\
\end{pmatrix}
\end{equation}
is a contraction from $\ell_2(\mathbb N_0,\mathbb C^t)$ into $\ell_2(\mathbb N_0,\mathbb C^r)$.
\label{toto123}
\end{thm}

\begin{proof}
We set $r=t=1$ to simplify the arguments.
Assume first that $s$ is a contraction, and let $P$ denote the orthogonal projection from $\mathbf L_2(\mathbb T)$ onto $\mathbf H_2(\mathbb T)$. Then the Toeplitz operator
$f\mapsto Ps^*f$ is a contraction from $\mathbf H_2(\mathbb T)$ into $\mathbf H_2(\mathbb T)$. It admits thus a matrix representation. Using the basis $1,z,z^2,\ldots$ we see that
\[
  \langle Ps^*z^n,z^m\rangle_2=\begin{cases}\,\,0,\hspace{7mm}\,\,\, n<m,\\
    \,\, s_{n-m},\,\ n\ge m\end{cases}
\]
and hence the Toeplitz matrix representation. For the converse, we assume that $T_s$ is a contraction. We compute $T_s^*e_z$ where $e_z=(1,z,z^2,\ldots)^t$. We have (compare with \eqref{toep})
\begin{equation}
\label{toeplitz12345}
T_s^*e_z=\overline{s(z)}e_z
\end{equation}
and hence the result.
\end{proof}
\section{Positive operators, reproducing kernel spaces and Toeplitz operators}
\setcounter{equation}{0}
We use various tools from quaternionic operator theory and in particular from the theory of linear relations and the theory of reproducing kernel spaces,
as developed in \cite{zbMATH06658818}. We recall:

\begin{defn}
  \label{real1!!!}
Given two right (or left, or two-sided) linear spaces $\mathcal V,\mathcal W$ over the quaternions,
a linear relation is a linear subspace of the Cartesian product $\mathcal V \times \mathcal W$.
\end{defn}

The graph of a (possibly not everywhere defined) linear operator is a linear relation, but there are linear relations which are not graphs of operators.\\

We will define inner products on a quaternionic right vector space, say $\mathcal V$, with the following convention
\begin{equation}
\langle fu,gv\rangle=\overline{v}\langle f,g\rangle u,\quad f,g\in\mathcal V,\,\,\, u,v\in\mathbb H
\end{equation}
and satisfying moreover
\begin{equation}
\langle f\, ,\, ug\rangle=\langle\overline{u}f\, ,\, g\rangle,\quad f,g\in\mathcal V,\,\,\, u\in\mathbb H ,
\end{equation}
when the quaternionic space under study is two-sided (for instance, $\ell_2(\mathbb N_0,\mathbb H)$).\smallskip

Let $K(x,y)$ be a the $\mathbb H^{r\times r}$-valued function, positive definite on $\Omega$. We will
denote by $\mathcal H(K)$ the reproducing kernel space of $\mathbb H^r$-valued functions with reproducing kernel $K$.\\

Let $K_1(z,w)$ and $K_2(z,w)$ be two $\mathbb H^{r\times r}$-valued functions positive definite on a set $\Omega$. We recall that $K_1\le K_2$ means that the
difference $K_2-K_1$ is still positive definite in $\Omega$.
This happens if and only if the space $\mathcal H(K_1)$ is contractively included in the space $\mathcal H(K_2)$.\\

The following result, relating operator ranges and reproducing kernel
Hilbert spaces is well known. See \cite{MR4126757} for a discussion in the quaternionic and indefinite inner product setting.

\begin{prop}
Let $\Gamma$ be a positive bounded operator from the left quaternionic Hilbert space $\mathcal H$ into itself. Let $\Omega$ be a set and let
$z\mapsto f_z$ be a $\mathcal H$-valued function defined on $\Omega$, and such that the closed left-linear span of the vectors $f_w$ is equal to $\mathcal H$. The function
\begin{equation}
\label{k12345}
K(z,w)=\langle \Gamma f_w,f_z\rangle
\end{equation}
is positive definite on $\Omega$ and the associated reproducing kernel Hilbert space with reproducing kernel consists of the functions of the form
\[
F(z)=\langle \sqrt{\Gamma}f,f_z\rangle,\quad f\in\mathcal H,
\]
with norm
\[
\|F\|=\|(I-\pi)f\|
\]
where $\pi$ is the orthogonal projection onto the kernel of $\Gamma$.
\end{prop}

Let us set $\mathcal H=\ell_2(\mathbb N_0,\mathbb H^r)$ in the previous proposition. Since $\Gamma$ is bounded, it has a block matrix representation
$\Gamma=(\Gamma_{nm})$, where  $\Gamma_{mn}\in\mathbb H^{r\times r}$.  We can write
\begin{equation}
\langle \Gamma f,g\rangle=\sum_{n,m=0}^\infty g_n^*\Gamma_{nm}f_m,
\end{equation}
and
\[
K(z,w)=\sum_{n,m=0}^\infty f_n(z)^*\Gamma_{nm}f_m(w).
\]
Cases of interest in the present work are:
\begin{equation}
\label{case1}
\Omega\subset\mathbb R^4\cong\mathbb H,\quad \text{{\rm and\,\, denoting\,\ $z=p$}},\quad f_n(p)=\overline{p}^nI_r
\end{equation}
and
\begin{equation}
\label{case11}
\Omega\subset \mathbb R^4\cong\mathbb H,\quad \text{{\rm and\,\, denoting\,\ $z=x$}},\quad f_n(x)=\overline{P_n(x)}I_r.
\end{equation}

Assume now $\Gamma$ to be of the form
\begin{equation}
\label{gamma123}
\Gamma =I-T_ST_S^*
\end{equation}
with $T_S$ as in \eqref{Tesse}
and where $S_i\in\mathbb H^{r\times t}$, $i=0,1,\ldots$. In particular, the block Toeplitz operator $T_S$ is a contraction.
The kernel becomes in the first case
\begin{equation}
  \sum_{n=0}^\infty p^n\overline{q}^nI_r-\sum_{n=0}^\infty\left(\sum_{m=n}^\infty p^mS_{m-n} \right)\left(\sum_{m=n}^\infty q^mS_{m-n}\right)^*,
\end{equation}
and
\begin{equation}
  \sum_{n=0}^\infty P_n(x)\overline{P_n(y)}I_r-\sum_{n=0}^\infty\left(\sum_{m=n}^\infty P_m(x)S_{m-n} \right)\left(\sum_{m=n}^\infty P_m(y)S_{m-n} \right)^*
\end{equation}
in the second case.\\

We note that, with the $\star$-product (see \cite{zbMATH06658818}):
\begin{equation}
\label{123ert}
\sum_{m=n}^\infty p^mS_{m-n} =p^n\star \left(\sum_{m=0}^\infty p^mS_{m} \right)
\end{equation}
and similarly, with the $CK$-product, using \eqref{nmk} in Lemma \ref{PnPm},
\begin{equation}
\label{facto-pn}
\sum_{m=n}^\infty P_m(x)S_{m-n} =P_n(x)\odot\left(\sum_{m=0}^\infty P_m(x)S_{m} \right) .
\end{equation}
The functions
\[
\sigma(p)=\sum_{m=0}^\infty p^mS_{m}
\]
and
\[
S(x)=\sum_{m=0}^\infty P_m(x)S_{m}
\]
are Schur multipliers, for the slice hyperholomorphic and for the present case (called Appell-like case), respectively.

\section{The Hardy space and Schur multipliers}
\setcounter{equation}{0}
In this section we will introduce and study the Hardy space in this framework. To start with, we denote by $\mathcal E$ the ellipsoid
\begin{equation}
\mathcal E=\left\{x\in\mathbb R^4\,:\, 9x_0^2+x_1^2+x_2^2+x_3^2<1\right\}
\end{equation}
and we prove the following:
\begin{lem}
The function
\begin{equation}
\label{hardy-here}
k_{\mathcal E}(x,y)=\sum_{m=0}^\infty P_1^{m\odot}(x)\overline{P_1^{m\odot}(y)}
\end{equation}
converges and is positive definite for $x,y\in\mathcal E$.
\end{lem}
\begin{proof}
For $x\in\mathcal E$ we have $|P_1(x)|<1$ and the result follows from Corollary \ref{coro234}.
\end{proof}

We point out that using \eqref{ze1ze2ze3}, we get
\begin{equation}
\label{hardy-here-2}
k_{\mathcal E}(x,y)=\sum_{m=0}^\infty\left(\sum_{|\nu|=m}\zeta(x)^{\nu}{\mathbf e}^{\nu}\frac{|\nu|!}{\nu!}\right)\overline{\left(\sum_{|\mu|=m}\zeta(y)^{\mu}{\mathbf e}^\mu\frac{|\mu|!}{\mu!}\right)}.
\end{equation}
\begin{rem}
In \cite{asv-cras,MR2124899,MR2240272} a different approach was used and a similar construction yields the Drury-Arveson kernel
\begin{equation}
\label{drury}
\begin{split}
K(x,y)&=\sum_{m=0}^\infty\sum_{|\nu|=m}\frac{|\nu|!}{\nu!}\zeta(x)^{\nu}\overline{\zeta(y)^{\nu}}\\
&=(1-\zeta_1(x)\overline{\zeta_1(y)}  -\zeta_2(x)\overline{\zeta_2(y)}-\zeta_3(x)\overline{\zeta_3(y)})^{-\odot}.
\end{split}
\end{equation}
Note that the formula \eqref{hardy-here} is easier to work with than formula \eqref{hardy-here-2}. We also note that
\begin{equation}
k_{\mathcal E}(x_0,y_0)=\frac{1}{1-9x_0y_0},\quad x_0,y_0 \in (-1/3,1/3).
\end{equation}
Using the polynomials $Q_n$ one can define
the kernel (see \cite[Remark 5.3]{ads-fh})
\[
K_Q(x,y)=\sum_{n=0}^\infty Q_n(x)\overline{Q_n(y)}
\]
whose restriction to the real axis is different, indeed it is
\[
K_Q(x_0,y_0)=\frac{1}{1-x_0y_0},\quad x_0,y_0\in(-1,1).
\]
\end{rem}
\begin{defn}
The reproducing kernel Hilbert space associated with \eqref{hardy-here} will be called the Hardy space, and denoted by $\mathbf H_2({\mathcal E})$.
\end{defn}

\begin{thm}
The Hardy space $\mathbf H_2({\mathcal E})$ consists of functions of the form
\begin{equation}
f(x)=\sum_{m=0}^\infty\left(\zeta_1(x){\mathbf e}_1+\zeta_2(x){\mathbf e}_2+\zeta_3(x){\mathbf e}_3\right)^{m\odot}f_m=\sum_{m=0}^\infty P_m(x) f_m,
\end{equation}
where the coefficients $f_m$ belong to $\mathbb H$ and are such that
\begin{equation}
\sum_{m=0}^\infty |f_m|^2<\infty .
\end{equation}
This expression is then the square of the norm of $f$ in the Hardy space, i.e. $\|f\|^2=\sum_{m=0}^\infty |f_m|^2$.
\end{thm}

\begin{proof}
The proofs follows standard arguments, see \cite{ads-fh,as3}.
\end{proof}


From the form of the elements of the Hardy space $\mathbf H_2(\mathcal E)$ and using the fact that the polynomials $P_m$ are Fueter hyperholomorphic of axial type, see Remark 3.9 in \cite{ads-fh}, we deduce:

\begin{cor}
Elements of $\mathbf H_2(\mathcal E)$ are Fueter hyperholomorphic of axial type, in particular are uniquely determined by their restriction to $(-1/3,1/3)$.
\end{cor}
\begin{lem}
\label{lemma-fs}
The operator $\mathsf S\,:\, f\mapsto P_1\odot f$ is an isometry in the Hardy space, with adjoint given by
\begin{equation}
\label{bws}
\mathsf S^*\left(\sum_{n=0}^\infty P_nf_n\right)=\sum_{n=0}^\infty P_nf_{n+1}.
\end{equation}
Furthermore
\begin{equation}
\label{SS*}
\mathsf S \mathsf S^* f=f-f(0), \quad f\in \mathbf H_2(\mathcal E).
\end{equation}
\end{lem}

\begin{proof}
The proof is a consequence of
\[
\begin{split}
\mathsf S \mathsf S^* f&=P_1\odot\left(\sum_{n=0}^\infty P_nf_{n+1}\right)\\
&=\sum_{n=0}^\infty P_{n+1}f_{n+1}\\
&=f-f_0\\
&=f-f(0).
\end{split}
\]
\end{proof}

Let $Cf=f(0)$ be the point evaluation in $\mathbf H_2(\mathcal E)$. Then $C^*u=k_{\mathcal E}(\cdot, 0)u=u$ and we get from the previous lemma
\begin{equation}
I-M_{P_1}M_{P_1}^*=C^*C.
\label{fundamental987}
\end{equation}

This equation is really what makes the arguments work in the sequel, and in particular in the construction of a coisometric realization.

\begin{defn}
The operator \eqref{bws} will be called the backward-shift operator and denoted by $R_0$.
\end{defn}

\begin{ex}
{\rm Let $a\in\mathcal E$. The space of functions in the Hardy space such that $f(a)=0$ need not be $\mathsf S$-invariant. On the
other hand, the space of functions $f\in\mathbf H_2(\mathcal E)$ such that
\begin{equation}
(P_n\odot f)(a)=0,\quad n=0,1,2,\ldots
\label{general-inter}
\end{equation}
is $\mathsf S$-invariant, see Lemma \ref{PnPm}. This suggests that the natural homogeneous interpolation condition is \eqref{general-inter} and not merely $f(a)=0$. See \cite[p. 148]{MR2124899} for a related remark.}
\end{ex}
We can consider hyperholomorphic functions operator-valued, in particular matrix or vector-valued. The definition of this class of functions is given by following the classical complex case, but we repeat it for the sake of completeness.
\begin{defn} Let $\mathcal{X}$ be a two-sided quaternionic Banach space, and $\mathcal{X}^*$ be its dual. A function $f: U\subset\mathbb{H} \to \mathcal{X}$, where $U$ is open, is said to be weakly (left) hyperholomorphic in $U$ if $\Lambda f$ satisfies $D(\Lambda f)=0$ for every $\Lambda\in\mathcal{X}^*$.
\end{defn}
We recall that a function is hyperholomorphic if and only if it is differentiable in a suitable sense, see \cite[Theorem 3]{MR1048702} and we follow this notion of differentiability to state the following definition, in which we identify $\mathbb H$ with $\mathbb H_3=\{\vec{\zeta}=(\zeta_1,\zeta_2,\zeta_3)\ | \zeta_i=x_i-\mathbf{e}_i x_0, \, i=1,2,3\}$ as a real linear space via the map $(\zeta_1,\zeta_2,\zeta_3)\mapsto \zeta_1\mathbf{e}_1+\zeta_2\mathbf{e}_2+\zeta_3\mathbf{e}_3$:
\begin{defn}
Let $\vec{a}\in\mathbb H_3$, $U$ be a neighborhood of $\vec{a}$ and let $F:\, U\to \mathcal{X}$ be a continuous function. Then $f$ is called left (resp. right) strongly differentiable in $\vec a$ in the quaternionic sense if there exists a left (resp. right) linear map $L:\ \mathbb H_3\to\mathcal{X}$ such that
\begin{equation}\label{diff}
\lim_{\Delta \vec{z}\to 0}\dfrac{\|f(\vec{a+\Delta \vec{z}})-f(\vec{a})-L(\Delta \vec{z})\|_{\mathcal X}}{\| \Delta \vec{z}\|}=0
\end{equation}
where $\| \vec{z}\|=\sum_{i=1}^3 \overline{\zeta_i}\zeta_i$. A function is strongly differentiable in $U$ if it is so at every point $\vec a\in U$.
\end{defn}
The definition originally considered by Malonek in \cite{MR1048702} can be obtained from the previous one when $\mathcal{X}=\mathbb{H}$. Since, in the scalar case, the definition is also equivalent to that one of left (resp. right) hyperholomorphy, we will equivalently say that a function $f$ as in Definition \ref{diff} is strongly hyperholomorphic. See also \cite{APS2020} for a theory of hyperholomorphic functions whose values are taken in a Banach algebra.
Using the same arguments as in the complex case, see \cite[Theorem VI.4]{MR751959}, which are valid also in the quaternionic case, see \cite{zbMATH06658818}, one can prove:
\begin{thm}
A function is weakly hyperholomorphic in $U$ if and only if it is strongly hyperholomorphic in $U$.
\end{thm}
The validity of this result allows to simplify the terminology and we shall say, for short, that $f$ is hyperholomorphic with values in $\mathcal{X}$.
In the special case in which $\mathcal{X}=\mathbb{H}^{r\times s}$, a function is weakly hyperholomorphic if and only if all its entries are left or right hyperholomorphic there.

\medskip
The next result was proved in the quaternionic setting  in the context of slice hyperholomorphic functions, see e.g. \cite[Section 7]{zbMATH06658818}; here we prove its counterpart in the
present framework.
\begin{thm}
Let $\mathfrak M$ be a finite dimensional linear right-vector space of $\mathbb H^{r}$-valued functions, and hyperholomorphic of axial type in a neighborhood of the origin.
Then $\mathfrak M$ is $R_0$-invariant
if and only if there exists a pair of matrices
$(\mathsf C,\mathsf A)\in\mathbb H^{r\times N}\times \mathbb H^{N\times N}$ such that $f\in\mathfrak M$ if and only if it can be written as
\begin{equation}
\label{formsystem}
f=\sum_{n=0}^\infty P_n\mathsf C\mathsf A^n\xi,\quad \xi\in\mathbb{H}^N.
\end{equation}
We have $N\ge {\rm dim}\, \mathfrak M$, and there is equality if and only if the pair $(C,A)$ is observable, meaning
\begin{equation}
\cap_{n=0}^\infty \ker \mathsf C\mathsf A^n=\left\{0\right\}.
\end{equation}
\end{thm}
\begin{proof}
Let $f_1,\ldots, f_N$ be a basis of $\mathfrak M$, and let $F=\begin{pmatrix} f_1&f_2&\cdots& f_N\end{pmatrix}$. Let $F=\sum_{n=0}^\infty P_nF_n$, with $F_n\in\mathbb H^{r\times N}$.
In view of the $R_0$-invariance there exists a matrix $\mathsf A\in\mathbb H^{N\times N}$ such that
\[
\sum_{n=0}^\infty P_n F_{n+1}=\left(\sum_{n=0}^\infty P_n F_{n}\right)\mathsf A.
\]
It follows that
\[
F_{n+1}=F_n\mathsf A,\quad n=0,1,\ldots
\]
and \eqref{formsystem} follows with $F_0=\mathsf C$.\smallskip

The last claim follows from the fact that,
\[
f\equiv 0\,\,\iff\,\, \mathsf C\mathsf A^n\xi=0,\quad n=0,1,2,\ldots
\]
\end{proof}

Note that, for $x_1=x_2=x_3=0$, we have
\[
f(x_0)=\mathsf C(I_N-3x_0\mathsf A)^{-1}\xi.
\]
Since the $CK$-product is not a law of composition we cannot, a priori, define Schur multipliers (see Definition \ref{schurmult}) in terms of multiplication operators. We define them in terms of positive definite functions.
The corresponding contractive operator, counterpart of the $CK$-multiplication by $S$, is given in Proposition \ref{prop-6-1} and Definition \ref{def-6-2}; see equation \eqref{ms-321}.

\begin{defn}\label{schurmult}
The $\mathbb H^{r\times s}$-valued hyperholomorphic function $S$ is called a Schur multiplier if the kernel
\begin{equation}
\label{tyu}
K_S(x,y)=\sum_{n=0}^\infty \left(P_n(x)\overline{P_n(y)}I_r-(P_n\odot S)(x)((P_n\odot S)(y))^*\right)
\end{equation}
is positive definite in $\left\{x\in \mathbb R^4\,:\; 9x_0^2+x_1^2+x_2^2+x_3^2<1\right\}$.
\end{defn}
\begin{ex}
For instance $S=P_1I_r$ is a Schur multiplier since
\[
\begin{split}
  \sum_{n=0}^\infty \left(P_n(x)\overline{P_n(y)}I_r-(P_n\odot S)(x)((P_n\odot S)(y))^*\right)&=    \\
  &\hspace{-4cm}=
  \sum_{n=0}^\infty \left(P_n(x)\overline{P_n(y)}I_r-(P_n\odot P_1)(x) \overline{(P_n\odot P_1)(y)}I_r\right)\\
  &  \hspace{-4cm}= \sum_{n=0}^\infty P_n(x)\overline{P_n(y)}I_r-\sum_{n=0}^\infty P_{n+1}(x)\overline{P_{n+1}(y)}I_r\\
         &\hspace{-4cm}=I_r.
\end{split}
\]
\end{ex}
This example is of course quite trivial. We will give in Section \ref{co} a complete characterization of Schur multipliers, from which one can get numerous other examples.\\

The positivity of the kernel \eqref{tyu} is equivalent to the contractive inclusion of the reproducing kernel Hilbert space with reproducing kernel
\begin{equation}
\sum_{n=0}^\infty(P_n\odot S)(x)((P_n\odot S)(y))^*
\label{tyu1}
\end{equation}
inside the Hardy space. In particular, if $S$ is a Schur multiplier $P_n\odot S\in\mathbf H_2(\mathcal{E})$ for every $n$.

\begin{thm}
  \label{toto}
The $\mathbb H^{r\times s}$-valued function $S$ is a Schur multiplier if and only if the
lower triangular  Toeplitz operator
\begin{equation}
\label{toeplitz}
\mathscr T=\begin{pmatrix} S_0&0&0&\cdots&\\
  S_1&S_0&0&\cdots&\\
  S_2&S_1 &S_0 &0&&\\
 \vdots &\ddots &\ddots &\ddots &\\
    & & &&
\end{pmatrix}
\end{equation}
is a contraction from $\ell_2(\mathbb N_0,\mathbb H^s)$ into $\ell_2(\mathbb N_0,\mathbb H^r)$.
\end{thm}

\begin{proof}
Assume first that the kernel \eqref{tyu} is positive definite in $\mathcal{E}$. We divide the argument in a number of steps.\\

STEP 1: {\sl There exist  coefficients $S_0,S_1,\ldots\in\mathbb H^{r\times s}$ such that
\[
S(x) =\sum_{u=0}^\infty P_u(x) S_u
\]
and $\sum_{u=0}^\infty \|S_u\|^2<\infty$.}\smallskip

Indeed, let $K_2(x,y)=\sum_{n=0}^\infty P_n(x)I_r\overline{P_n(y)}$ and $K_1(x,y)=\sum_{n=0}^\infty
  (P_n\odot S)(x)((P\odot S)(y))^*$. The inclusion operator
\[
\mathcal{I}\left(\sum_{n=0}^\infty (P_n\odot S)f_n\right)=\sum_{n=0}^\infty (P_n\odot S) f_n
\]
is a contraction from $\mathcal H(K_1)$ inside $\mathcal H(K_2)$, and so in particular $S\in(\mathbf H_2(\mathcal E))^{r\times s}$.\\

STEP 2: {\sl The function
\[
s(a)=\sum_{n=0}^\infty\chi(S_n)a^n,\quad a\in(-1,1),
\]
has an analytic contractive extension to the open unit disk of the complex plane.}\\

We write for $x_0,y_0\in(-1/3,1/3)$
\[
K_s(3x_0,3y_0)=\frac{I_r-s(3x_0)s(3y_0)^*}{1-9x_0y_0}
 \]
(of course, $K_s(3x_0,3y_0)$ does not characterize in a unique way $K_s(x,y)$).
The kernel $K_s(3x_0,3y_0)$ is positive definite for $x_0,y_0\in(-1/3,1/3)$. We set $a=3x_0$ and $b=3y_0$. The kernel
\[
\frac{I_r-s(a)s(b)^*}{1-ab}
\]
is positive definite in $(-1,1)$, and so is the kernel $\chi\left(\dfrac{I_r-s(a)s(b)^*}{1-ab}\right)$, and we conclude by applying Theorem \ref{dono}.\\

STEP 3: {\sl The Toeplitz operator based on the sequence $\chi(S_u)$ is contractive from
$\ell_2({\mathbb N}_0,\mathbb C^{2s})$ into $\ell_2({\mathbb N}_0,\mathbb C^{2})$}\smallskip

This follows from Theorem \ref{toto123}.\\

STEP 4: {\sl The Toeplitz operator $T_s$ is contractive.}\smallskip

This follows from Proposition \ref{2-9-0}. We restrict the operator in Step 3 to sequences of matrices in the  range of $\chi$.\\

We now suppose that $\mathscr T$ is a contractive operator. We write
\[
\langle (I-\mathscr T\mathscr T^*)c,c\rangle_{\ell_2(\mathbb N_0,\mathbb H^r)}\ge 0
\]
with
\begin{equation}
\label{cccc}
c=\sum_{i} \begin{pmatrix} I_r\\ \\ \overline{P_1(x^{(i)})}I_r\\ \\ \overline{P_2(x^{(i)})}I_r\\ \vdots\end{pmatrix}u_i
\end{equation}
to get the positivity of the kernel $K_S$.
\end{proof}

\begin{cor}
In the above notation, the function
\begin{equation}
\sum_{n=0}^\infty p^n S_n
\end{equation}
is a slice hyperholomorphic Schur multiplier.
\end{cor}

Given two multipliers, the bounded operator $M_{S_1}M_{S_2}$ will not be a multiplier in general.

\begin{prop}
\label{prop-6-1}
Let $S$ be a $\mathbb H^{r\times s}$-valued Schur multiplier. The formula
\begin{equation}
T\left(\sum_{n=0}^\infty P_n\overline{P_n(a)}u\right)=\sum_{n=0}^\infty P_n\left((P_n\odot S)(a)\right)^*u,\quad a\in\mathcal E,\,\,u\in\mathbb H^r,
\label{t-123}
\end{equation}
defines a contraction from $(\mathbf H_2(\mathcal E))^s$ into $(\mathbf H_2(\mathcal E))^r$, with adjoint given by
\begin{equation}
  \label{adjointm}
T^*\left(\sum_{n=0}^\infty P_nu_n\right)=\sum_{n=0}^\infty (P_n\odot S)u_n,\quad u_n\in\mathbb H^s.
\end{equation}
\end{prop}

\begin{proof} Let $\sum_{n=0}^\infty P_nu_n\in\mathbf H_2(\mathcal E)$. We can write:
\[
\begin{split}
\left\langle T^*\left(\sum_{n=0}^\infty P_nu_n\right), \sum_{n=0}^\infty P_n\overline{P_n(b)}u\right\rangle&=
\left\langle \sum_{n=0}^\infty P_nu_n, T\left(\sum_{n=0}^\infty P_n\overline{P_n(b)}u\right)\right\rangle\\
\\
&=  \left\langle \sum_{n=0}^\infty P_nu_n, \sum_{n=0}^\infty P_n((P_n\odot S)(b))^*u\right\rangle\\
\\
&=u^*\sum_{n=0}^\infty (P_n\odot S)(b)u_n.
\end{split}
\]
\end{proof}

\begin{rem}{\rm Formula \eqref{adjointm} gives the adjoint of $CK$-multiplication for any (bounded) multiplier (i.e. functions for which the corresponding operator of $CK$-multiplication
    is bounded in the Hardy space), and not only for Schur multipliers.}
  \label{new}
  \end{rem}

\begin{cor}
  Let $S(x)=\sum_{n=0}^\infty P_n S_n$ be a $\mathbb H^{r\times t}$-valued Schur multiplier. Then,
  \begin{equation}
    \widetilde{S}(x)=\sum_{n=0}^\infty P_n(x)S_n^*
  \end{equation}
  is a $\mathbb H^{t\times r}$-valued Schur multiplier.
\end{cor}

\begin{proof}
  $S$ is a Schur multiplier if and only if the function $\sum_{n=0}^\infty z^n\chi(S_n)$ is analytic and contractive in $\mathbb D$; this will hold if and only if
  the function
  \[
\sum_{n=0}^\infty z^n\chi(S_n)^*=\sum_{n=0}^\infty z^n\chi(S_n^*)
\]
is analytic and contractive in $\mathbb D$; by the previous theorem this will hold if and only if $\widetilde{S}$ is a $\mathbb H^{t\times r}$-valued Schur multiplier.
\end{proof}

The operator $T^*$ cannot be written as $S\odot\left(\sum_{n=0}^\infty P_nu_n\right)$, i.e. it is not
the $\odot$ multiplication by $S$. We now introduce the counterpart of this latter operator here.
\begin{defn}\label{MS}
The operator $T^*$ will be denoted by $M_S$:
\begin{equation}
\label{ms-321}
M_S\left(\sum_{n=0}^\infty P_n(a)u_n\right)=\sum_{n=0}^\infty (P_n\odot S)(a)u_n,\quad a\in\mathcal E.
\end{equation}
\label{def-6-2}
\end{defn}


We have the following extension result, counterpart of Theorem \ref{dono}. The proof is slightly different.

\begin{thm}
\label{ttt*}
Assume the kernel \eqref{tyu} defined and positive definite in a neighborhood $N$ of the origin of $\mathbb R^4$. Then, $S$ extends, in a unique way,  to a Schur multiplier.
\end{thm}

\begin{proof} We consider the scalar case to simplify the notation.
The preceding argument still holds and, setting $a=0$ and $b\in N$, we get
\[
(T^*v)(b)=S(b)v,\quad  b\in N.
\]
But $T^*v\in\mathbf H_2(\mathcal E)$ and in particular is hyperholomorphic in all of $\mathcal E$. More generally, still for $b\in N$, but for $f=\sum_{n=0}^\infty P_nf_n\in\mathbf H_2(\mathcal E)$
we have
\[
\begin{split}
(T^*f)(b)&=\langle (T^*f)(\cdot), k_{\mathcal E}(\cdot, b)\rangle_2\\
&=\langle f(\cdot), Tk_{\mathcal E}(\cdot, v)\rangle_2\\
&=\sum_{n=0}^\infty (P_n\odot S)(b)f_n.
\end{split}
\]
Writing $(T^*f)(b)=\sum_{n=0}^\infty P_nh_n$ and $S=\sum_{n=0}^\infty P_ns_n$, the previous equality is equivalent (since $v$ varies in an open set; it would be enough to have an interval such that
$x_0\in(-\epsilon,\epsilon)$)
\[
\begin{pmatrix}
  h_0\\ h_1\\ \vdots\end{pmatrix}=T_S  \begin{pmatrix}f_0\\f_1\\ \vdots\end{pmatrix},
\]
where $T_S$ is the lower triangular Toeplitz operator based on the coefficients of $S$. So $T_S$ is a contraction, and so
$I-T_ST_S^*\ge 0$, and so $\langle(I-T_ST_S^*c,c\rangle\ge 0$ for every $c\in\ell_2(\mathbb N_0,\mathbb H)$. The choice \eqref{cccc} allows to conclude that $S$ is a Schur multiplier.
\end{proof}

\begin{lem} The following equality holds for all $    f\in\mathbf H_2(\mathcal E)$:
\begin{equation}
\label{PS}
P_1\odot (M_Sf)=M_S(P_1\odot f).
\end{equation}
\end{lem}

\begin{proof}
Let $f=\sum_{n=0}^\infty P_nu_n$. We can write:
\[
\begin{split}
P_1\odot (M_Sf)&=P_1\odot\left(\sum_{n=0}^\infty (P_n\odot S)u_n\right)\\
&=\sum_{n=0}^\infty (P_{1+n}\odot S)u_n\\
&=M_S\left(\sum_{n=0}^\infty P_{n+1}u_n\right)\\
&=M_S\left(P_1\odot f\right)\\
&=M_S\left(M_{P_1} f\right).
\end{split}
\]
\end{proof}

It is useful to rewrite \eqref{PS} as
\begin{equation}
\label{PS1212}
M_{P_1}M_S=M_SM_{P_1}.
\end{equation}

\section{Schur algorithm}
\setcounter{equation}{0}
The Schur algorithm is based on Schwarz lemma and on the fact that if two numbers $u$ and $v$ are in the open unit disk so is $\dfrac{u-v}{1-u\overline{v}}$. It reads
(see \cite{schur,schur2}):

\begin{thm}
Let $f$ be analytic and contractive in the open unit disk (i.e. a Schur function), and assume $|f(0)|<1$, and set $f^{(0)}=f$.  Then the recursion
\begin{equation}
\label{fn}
f^{(n+1)}(z)=\begin{cases}\,\,\dfrac{f^{(n)}(z)-f^{(n)}(0)}{z(1-\overline{f^{(n)}(0)}f^{(n)}(z))},\quad 0<|z|<1,\\
  \\
 \,\, \dfrac{(f^{(n)})^\prime(0)}{1-|f^{(n)}(0)|^2},\quad \hspace{12.5mm}z=0,\end{cases}
\end{equation}
defines a family of Schur functions; it stops at rank $n_0$ if $|f^{(n_0)}(0)|=1$.
\end{thm}
The numbers $\rho_n=f^{(n)}(0)$ are called the Schur parameters associated with the Schur function $f$.\smallskip

This recursion cannot be considered directly in the matrix-valued case.  One needs to take into account that if $E_1$ and $E_2$ are strictly contractive matrices, say in
$\mathbb C^{p\times q}$, the matrix $(E_1+E_2)(I_q+E_1^*E_2)^{-1}$ need not be contractive, but the matrix
\begin{equation}
(I_p-E_1E_1^*)^{1/2}(E_2+E_1)(I_p+E_1^*E_2)^{-1}(I_q-E_1^*E_1)^{1/2}
\end{equation}
is strictly contractive.\\

The matricial  Schur algorithm was studied in \cite{dgk1} and, in the next result, we repeat the statement taking into account the matrix symmetry
\begin{equation}\label{sym2}
E_p^{-1}\overline{M}E_q=M, \qquad E_m=\left(\begin{matrix} 0 & I_m\\ -I_m & 0\end{matrix}\right),\ m=p,q
\end{equation}
is in force.
\begin{thm}
  Let $s$ be a $\mathbb C^{p\times q}$-valued Schur function satisfying  \eqref{sym2}. Assume $s_0$ strictly contractive.
  Then  the function
  \begin{equation}
    \label{schurmat}
    s^{(1)}(z)=\begin{cases}\frac{1}{z}(I_p-s_0s_0^*)^{-1/2}(s-s_0)(I_q-s_0^*s)^{-1}(I_q-s_0^*s_0)^{1/2},\quad 0<|z|<1,
      \\
      (I_p-s_0s_0^*)^{-1/2}s^{\prime}(0)(I_q-s_0^*s_0)^{-1/2},\quad z=0,\end{cases}
\end{equation}
is a Schur function and satisfies the symmetry \eqref{sym2}.
\end{thm}

If $\|s_0\|<1$ one can iterate, and one gets the matricial Schur algorithm.\\

The condition $\|s(0)\|<1$ is quite restrictive. A tangential Schur algorithm was developed in \cite{ad3}. On the other hand, when $p=q=2$ and
$s(0)$ is in the range of $\chi$ both $(I_2-s_0s_0^*)^{-1/2}$ and $(I_2-s_0^*s_0)^{1/2}$ are scalar matrices and \eqref{schurmat} reduces to

  \begin{equation}
    \label{schurmat123321}
s^{(1)}(z)=\frac{1}{z}(s(z)-s_0)(I_2-s_0^*s(z))^{-1}.
\end{equation}

We now turn to the setting of hyperholomorphic functions of axial type. For simplicity of exposition we first consider scalar valued Schur multiplier.
From the analysis in the previous sections, $S=\sum_{n=0}^\infty P_nS_n$
is a Schur multiplier if and only if the block Toeplitz
operator
\begin{equation}
\begin{pmatrix} \chi(S_0)&0&0&\cdots\\
  \chi(S_1)&\chi(S_0)&0&\cdots\\
  \chi(S_2)&\chi(S_1) &\chi(S_0) &0&\\
  & \cdots & &\cdots \\
\end{pmatrix}
\end{equation}
is a contraction from $\ell_2(\mathbb N_0,\mathbb C^2)$ into $\ell_2(\mathbb N_0,\mathbb C^2)$. The function
\[
s(z)=\sum_{n=0}^\infty\chi(S_n)z^n
\]
takes then contractive values, and is in the range of $\chi$; see Lemma \ref{rangechi}.

\begin{thm}
  Let $S$ be a quaternionic Schur multiplier  such that $|S(0)|<1$. Then the function
  \[
  S^{(1)} (3x_0)=\frac{1}{3x_0}(S(3x_0)-S(0))(1-\overline{S(0)}S(3x_0))^{-1}
\]
extends to a Schur multiplier.
\end{thm}

In the matrix-valued case it is not true anymore that $I-\chi(S_0)\chi(S_0)^*$ is a scalar matrix.

\begin{thm}
  Let $S$ be a $\mathbb H^{r\times t}$-valued Schur multiplier, and assume $\|S(0)\|<1$.
  The function
  \[
  S^{(1)}(3x_0)=\frac{1}{3x_0}(I_r-S_0S_0^*)^{-1/2}(S(3x_0)-S_0)(I_t-S_0^*S(3x_0))^{-1}(I_t-S_0^*S_0)^{1/2},
\]
with $ x_0\in(-1/3,1/3)$ and $S_0=S(0)$ extends to a Schur multiplier.
\end{thm}

The question whether the tangential Schur algorithm developed in \cite{ad3} can lead to functions satisfying the required symmetry property in the matrix-valued case remains to be considered.

\section{Intrinsic functions}
\setcounter{equation}{0}
In this section we study quaternionic intrinsic Fueter hyperholomorphic functions. Let us recall that, given an hyperholomorphic function $f$ on some axially symmetric open set $\Omega$, we say that $f$ is quaternionic intrinsic if it satisfies the relation \begin{equation}\label{IntForm}
f(\overline{x})=\overline{f(x)}, \textbf{  } \forall x\in\Omega.
\end{equation}

\begin{prop}\label{PnInt}
The family of polynomials $( P_n)_{n\geq 0}$ consists of axially hyperholomorphic quaternionic intrinsic functions on $\mathbb{H}$.
\end{prop}
\begin{proof}
We know that for all $n\geq 0$ the polynomials $P_n$ are axially hyperholomorphic functions on $\mathbb{H}$. Furthermore, using the relation with the $n$-th quaternionic Appell polynomials $Q_n$, see \cite[(3.8)]{ads-fh}, we have
\[
\begin{split}  \displaystyle
\overline{P_n(x)}&=\frac{\overline{Q_n(x)}}{c_n}\\
&=\sum_{j=0}^n \frac{T_j^n}{c_n}\overline{x}^{n-j}x^j
\\
&=\frac{Q_n(\overline{x})}{c_n}
\\
&=P_n(\overline{x}).
\end{split}
\]
\end{proof}
\begin{prop}\label{IntRC}
Let $f$ be a hyperholomorphic function of axial type on some axially symmetric open set $\Omega$. Then, $f$ is quaternionic intrinsic if and only if it admits a power series representation with real coefficients with respect to the polynomials $(P_n)_{n\geq 0}$.
\end{prop}
\begin{proof}
We know by Theorem 3.10 in \cite{ads-fh} that $f$ admits a power series with respect to $(P_n)_{n\geq 0}$. So, we can write $f=\displaystyle \sum_{n=0}^\infty P_n f_n$ with $f_n\in\mathbb{H}$ for all $n\geq 0$. We assume that $f$ is intrinsic, thus the formula \eqref{IntForm} and Proposition \ref{PnInt} imply that
\[
\begin{split}  \displaystyle
\overline{f(x)}=f(\overline{x}), \forall x\in \Omega&\Leftrightarrow \sum_{n=0}^\infty \overline{P_n(x)f_n}=\sum_{n=0}^\infty P_n(\overline{x})f_n, \forall x\in\Omega \\
&\Leftrightarrow\sum_{n=0}^\infty \overline{f_n}\overline{P_n(x)}=\sum_{n=0}^\infty \overline{P_n(x)}f_n, \forall x\in \Omega
\\
&\Leftrightarrow\sum_{n=0}^\infty \overline{f_n}(3x_0)^n=\sum_{n=0}^\infty (3x_0)^nf_n, \forall x_0\in \mathbb{R}
\\
&\Leftrightarrow\overline{f_n}=f_n, \forall n\geq 0
\\
&\Leftrightarrow f_n\in \mathbb{R}, \forall n\geq 0.
\end{split}
\]
The equivalence between the second and the third lines holds because $P_n$ is the unique axially hyperholomorphic extension of $(3x_0)^n$. This ends the proof.
\end{proof}
\begin{prop}\label{CompositionLaw}
Let $S_1$ and $S_2$ be two hyperholomorphic functions of axial type, defined on some axially symmetric open set $\Omega$. If $S_1$ is quaternionic intrinsic, then $S_1\odot S_2$  admits a power series expansion with respect to the polynomials $(P_n)_{n\geq 0}$.
\end{prop}
\begin{proof}
We note that $S_1$ and $S_2$ have power series expansions in terms of $(P_n)_{n\geq 0}$ that we can write $\displaystyle S_1=\sum_{n=0}^\infty P_na_n$ and $\displaystyle S_2=\sum_{n=0}^\infty P_nb_n$. Since $S_1$ is quaternionic intrinsic we know by Proposition \ref{IntRC} that the coefficients $(a_n)_{n\geq 0}$ are real. Thus, we apply also Lemma \ref{PnPm} to get
\[
\begin{split}  \displaystyle
S_1\odot S_2&=\left(\sum_{n=0}^\infty P_na_n \right)\odot \left(\sum_{m=0}^\infty P_mb_m \right)\\
&=\sum_{n,m=0}^\infty (P_n\odot P_m)a_nb_m
\\
&=\sum_{n,m=0}^\infty P_{n+m}a_nb_m
\\
&=\sum_{n=0}^\infty P_{n}\left(\sum_{k=0}^na_kb_{n-k}\right).
\end{split}
\]

\end{proof}
\begin{prop}\label{MultiplicationOperator}
Let $S$ be a hyperholomorphic function of axial type. If $S$ is quaternionic intrinsic, then the operator $M_S$ coincides with the multiplication operator $f\mapsto S\odot f$.
\end{prop}
\begin{proof}
We note that since $S$ is quaternionic intrinsic, it has real coefficients. Thus, we have  $P_n\odot S=S\odot P_n$ for all $n\geq 0.$ Then, starting from Definition \ref{MS}, for any $\displaystyle f=\sum_{n=0}^\infty P_nu_n$, we have

\[
\begin{split}  \displaystyle
M_S(f)&=\sum_{n=0}^\infty (P_n\odot S )u_n \\
&=\sum_{n=0}^\infty (S\odot P_n )u_n
\\
&=S\odot \left(\sum_{n=0}^\infty P_{n}u_n\right)
\\
&=S\odot f.
\end{split}
\]
\end{proof}
\begin{prop}
Let $S_1$ and $S_2$ be two hyperholomorphic functions of axial type such that $S_1$ is quaternionic intrinsic. Then, we have \begin{equation}
M_{S_1}M_{S_2}=M_{S_1\odot S_2}.
\end{equation}
\end{prop}
\begin{proof}
We know by Proposition \ref{CompositionLaw} that $S_1\odot S_2$ is well defined and admits a power series expansion in terms of $(P_n)_{n\geq 0}$ since $S_1$ is intrinsic. Therefore, using Proposition \ref{MultiplicationOperator}, we have

\[
\begin{split}  \displaystyle
M_{S_1\odot S_2}(f)&=(S_1\odot S_2)\odot f \\
&=M_{S_1}(S_2\odot f)
\\
&= M_{S_1}M_{S_2}(f).
\\
&
\end{split}
\]
\end{proof}

\section{de Branges-Rovnyak spaces}
\setcounter{equation}{0}
The reproducing kernel Hilbert space with reproducing kernel \eqref{tyu} will be called the de Branges-Rovnyak space associated with the Schur multiplier $S$ and denoted by
$\mathcal H(S)$. The treatment using the Appell-like approach allows to prove results naturally extending the corresponding ones in the classical complex case. For example, we have the following characterization:

\begin{thm} We have
\begin{equation}
\left((I-M_SM_S^*)(k_{\mathcal E}(\cdot, a)u\right)(b)=K_S(b,a)u
\end{equation}
and
\begin{equation}
\mathcal H(S)={\rm Ran}\,\sqrt{I-M_SM_S^*}
\end{equation}
with inner product
\begin{equation}
\langle \sqrt{I-M_SM_S^*}f,\sqrt{I-M_SM_S^*}g\rangle_{\mathcal H(S)}=\langle (I-\pi)f,g\rangle_2
\end{equation}
where $\pi$ is the orthogonal projection on the kernel of $\sqrt{I-M_SM_S^*}$.
\end{thm}

\begin{proof}

The claims follow from \cite{MR4126757}.
\end{proof}

We recall the following, valid for $f,g\in\mathbf H_2(\mathcal E)$ (the first equality is a special case of the second one):
\begin{eqnarray}
  \label{in11}
\langle \Gamma f\,,\,\Gamma g\rangle_{\mathcal H(S)}&=&\langle \Gamma f\,,\, g\rangle_2,\\
  \langle \sqrt{\Gamma} f\, ,\,\Gamma g\rangle_{\mathcal H(S)}&=&\langle \sqrt{\Gamma} f\, ,\, \sqrt{\Gamma}g\rangle_2.
                                                                  \label{in22}
\end{eqnarray}

Using the quaternionic version of \cite{dbr2} or \cite[Theorem 4.1]{fw} (we do not give proofs of these since we will have more general results than the theorems below
in the next section) we have the following results, for matrix-valued Schur multipliers

\begin{thm} Let $S$ be a $\mathbb H^{r\times s}$-valued Schur multiplier.
An element $f$ in $\mathbf H_2(\mathcal E)$ belongs to $\mathcal H(S)$ if and only if
\begin{equation}
\sup_{g\in\mathbf H_2(\mathcal E)}\|f+M_Sg\|_2^2-\|g\|_2^2<\infty .
\end{equation}
\end{thm}
Using this characterization we can prove the following:
\begin{thm} Let $S$ be a $\mathbb H^{r\times s}$-valued Schur multiplier.
Let $R_0$ be defined by \eqref{bws}. Then:
\begin{equation}
\|R_0f\|_{\mathcal H(S)}^2\le \|f\|_{\mathcal H(S)}^2-\|f(0)\|^2,\quad f\in\mathcal H(S).
\end{equation}
\end{thm}

\begin{proof}
Recall that $\mathsf S$ denotes the forward-shift operator, and that the latter is an isometry (see Lemma \ref{lemma-fs}). Using \eqref{SS*} and \eqref{PS}
we can write for $f,g\in\mathbf H_2(\mathcal E)$:
\[
\begin{split}
\|R_0f+M_Sg\|_2^2-\|g\|^2_2&=\|\mathsf S(R_0f+M_Sg)\|_2^2-\|\mathsf S g\|^2_2\\
&=\|f-f(0)+M_S(P_1\odot g)\|^2_2-\|P_1\odot g\|_2^2\\
&=\|f+M_S(P_1\odot g)\|_2^2-2{\rm Re}\,\langle f+M_S(P_1\odot g)\,,\,  f(0)\rangle+\\
&\hspace{5mm}+\|f(0)\|^2 -\|P_1\odot g\|_2^2\\
&=\|f+M_S(P_1\odot g)\|_2^2-\|P_1\odot g\|_2^2 -\|f(0)\|^2,
\end{split}
\]
where we have used, with $g=\sum_{n=0}^\infty P_nb_n$,
\[
\begin{split}
\langle f+M_S(P_1\odot g)\,,\,  f(0)\rangle_2&=    \langle f+P_1\odot M_Sg\, ,\,f(0) \rangle_2\\
&=\langle f\, ,\, f(0)\rangle_2+\langle P_1\odot M_Sg\, ,\,f(0) \rangle_2\\
&=\|f(0)\|^2
\end{split}
\]
since $P_1\odot M_Sg$ has no constant term.
\end{proof}

\begin{thm}
Let $R_0$ be defined by \eqref{bws}. Then $R_0Su\in\mathcal H(S)$ for every $u\in\mathbb H^s$ and
\begin{equation}
\|R_0S\|^2\le \|u\|^2-\|S(0)u\|^2,\quad u\in\mathbb H^s.
\end{equation}
\end{thm}

\begin{proof}
We already know that $Su\in\mathbf H_2(\mathcal E)$. We write
\[
\begin{split}
\|R_0S+M_Sg\|_2^2-\|g\|^2_2&=\|\mathsf S(R_0S+M_Sg)\|_2^2-\|\mathsf S g\|^2_2\\
&=\|-S(0)u+M_S(u+P_1\odot g)\|^2_2-\|P_1\odot g\|_2^2\\
&=\|S(0)u\|^2+\|M_S(u+P_1\odot g)\|^2\\&  \hspace{2mm}-2{\rm Re}\, \langle S(0)\, ,\, M_S(u+P_1\odot g)\rangle\\
&\le \|S(0)\|_2+\|u+P_1\odot g\|_2^2 -\|P_1\odot g\|_2^2\\
&=\|u\|^2-\|S(0)u\|^2.
\end{split}
\]
In the above, to go from the third to fourth line we used that $$\|M_S(u+P_1\odot g)\|_2\le\|u+P_1\odot g\|_2$$ (since $S$ is a Schur multiplier) and
\[
\langle S(0)\, ,\, M_S(u+P_1\odot g)\rangle=0
\]
since $M_S(u+P_1\odot g)$ has no constant term in its expansion along the $P_n$. Similarly we used that $\langle u, P_1\odot g\rangle_2=0$ to go from the fourth to the last line.
\end{proof}

The operators defined in the previous theorems are part of a coisometric operator matrix. In \cite{dbr2} (see also \cite{dbjfa}) it is obtained using the theory of
complementation. In the next section we use a different method.
\section{The coisometric colligation and Blaschke functions}
\setcounter{equation}{0}
\label{co}
\subsection{The lurking isometry}
Let us denote by $\odot_r$ the right $CK$-product. Using \eqref{nmk} we note that \eqref{tyu} can be rewritten as
\begin{equation}
\label{lurking}
K_S(x,y)-P_1(x)\odot K_S(x,y)\odot_r \overline{P_1(y)}=I_r-S(x)S(y)^*,
\end{equation}
from which we get
\begin{equation}
\label{lurk1}
K_S(x,y)+S(x)S(y)^*=P_1(x)\odot K_S(x,y)\odot_r \overline{P_1(y)}+I_r.
\end{equation}
Write $K(x,y)=\langle f(x),f(y)\rangle_{\mathfrak H}$, where $\mathfrak H$ is a Hilbert space and $x\mapsto f(x)$ is $\mathfrak H$-valued function. We rewrite \eqref{lurk1}
as
\[
\left\langle\begin{pmatrix} S(y)^*h\\ f(y)\end{pmatrix}\,\,\begin{pmatrix} S(x)^*k\\ f(y)\end{pmatrix}\right\rangle_{\mathbb H^r\oplus\mathfrak H}\!\!\!\!\!=
\left\langle\begin{pmatrix} h\\ f(y)\odot_\ell \overline{P_1(y)}h\end{pmatrix}\,\,\begin{pmatrix}k\\ f(x)\odot_\ell \overline{P_1(y)}k\end{pmatrix}
\right\rangle_{\mathbb H^s\oplus\mathfrak H}
\]
where $h,k\in\mathbb H^s$ and $x,y\in\mathcal E$.\smallskip

This last equality, called the lurking isometry (see \cite{agler-hellinger,btv}), can be the tool to get a co-isometric realization of $S$ (see \cite{MR3303909} for an application in the
quaternionic setting). We will choose a different (and closely related) avenue, namely the theory of relations, which originates with the work of Krein and Langer
(the $\epsilon_z$ method; see e.g. \cite{kl1,MR47:7504}) and was developed further in \cite{adrs}. We will use the lurking isometry method in Section \ref{mult123} to characterize Carath\'eodory
and Schur multipliers  in the setting of the counterpart of Hardy space of the right half plane.
\subsection{The co-isometric realization}
We use the method of isometric relations, as in \cite{adrs}, suitably adapted to the present setting, and follow \cite[\S 2]{asv-cras}. We set $\Gamma=I-M_SM_S^*$, and define $w_a$ via
\[
w_aq=\Gamma M_{P_1}^*k_{\mathcal E}(\cdot, a)p\in\mathcal H(S),\quad q\in\mathbb H^r,
\]
and introduce:
\[
  \widehat{V}\begin{pmatrix}w_aq\\ u\end{pmatrix}=\begin{pmatrix} (K_S(\cdot, a)-K_S(\cdot, 0))q+K_S(\cdot, 0)u    \\  (S(a)^*-S(0)^*)q+S(0)^*u\end{pmatrix},\quad q\in\mathbb H^r,\,\,
  u\in\mathbb H^s.
\]

\begin{thm}
$\widehat{V}$ is isometric from its domain ${\rm Dom}\,\widehat{V}\subset\mathcal H(S)\oplus\mathbb H^r$ into $\mathcal H(S)\oplus\mathbb H^s$.
\end{thm}
\begin{proof}
We want to check:
\begin{equation}
  \left\langle \widehat{V}\begin{pmatrix}w_aq\\ u\end{pmatrix}\, ,\,\widehat{V}\begin{pmatrix}w_bp\\ v\end{pmatrix}\right\rangle_{\mathcal H(S)\oplus\mathbb H^s}=
  \left\langle \begin{pmatrix}w_aq\\ u\end{pmatrix}\, ,\,\begin{pmatrix}w_bp\\ v\end{pmatrix}\right\rangle_{\mathcal H(S)\oplus\mathbb H^r}
\end{equation}
where $p,q\in\mathbb  H^r$, $u,v\in\mathbb H^s$, and $a,b\in\mathcal E$.
We divide the proof into three  steps.\\

STEP 1: {\sl Case $u=v=0$.}\smallskip

Then, only terms involving the directions $p$ and $q$ appear.
In the following sequence of equality we use \eqref{PS1212} to go from the second to the third line, and \eqref{fundamental987} to go from the fourth to the fifth line. We also note that
\[
\begin{split}
\langle w_aq,w_bp\rangle_{\mathcal H(S)}&=\langle \Gamma M_{P_1}^*k_{\mathcal E}(\cdot, a)q\,,\, M_{P_1}^*k_{\mathcal E}(\cdot, b)p\rangle_2\\
&=  \langle  M_{P_1}^*k_{\mathcal E}(\cdot, a)q\,,\, M_{P_1}^*k_{\mathcal E}(\cdot, b)p\rangle_2\\&  \hspace{3cm}-\langle M_S^* M_{P_1}^*k_{\mathcal E}(\cdot, a)q\,,\, M_S^*M_{P_1}^*k_{\mathcal E}(\cdot, b)p\rangle_2\\
&=\langle  M_{P_1}^*k_{\mathcal E}(\cdot, a)q\,,\, M_{P_1}^*k_{\mathcal E}(\cdot, b)p\rangle_2\\&  \hspace{3cm}-\langle M_{P_1}^* M_S^*k_{\mathcal E}(\cdot, a)q\,,\, M_{P_1}^*M_S^*k_{\mathcal E}(\cdot, b)p\rangle_2\\
&=\langle  M_{P_1}M_{P_1}^*k_{\mathcal E}(\cdot, a)q\,,\, k_{\mathcal E}(\cdot, b)p\rangle_2\\&  \hspace{3cm}-\langle M_{P_1}M_{P_1}^* M_S^*k_{\mathcal E}(\cdot, a)q\,,\,M_S^* k_{\mathcal E}(\cdot, b)p\rangle_2\\
&=\langle  (I-C^*C)k_{\mathcal E}(\cdot, a)q\,,\, k_{\mathcal E}(\cdot, b)p\rangle_2\\&  \hspace{3cm}-\langle (I-C^*C)M_S^*k_{\mathcal E}(\cdot, a)q\,,\, M_S^*k_{\mathcal E}(\cdot, b)p\rangle_2.\\
\end{split}
\]
To pursue we note that
\[
C\left(k_{\mathcal E}(\cdot, a)q\right)=q
\]
and, using the formula \eqref{t-123} for $M_S^*$
\[
CM_S^*\left(k_{\mathcal E}(\cdot, a)q\right)=S(a)^*q.
\]
Thus
\[
\begin{split}
\langle w_aq,w_bp\rangle_{\mathcal H(S)}&=
p^*k_{\mathcal E}(b,a)q-p^*q+p^*S(b)S(a)^*q-\langle M_S^*k_{\mathcal E}(\cdot, a)q\,,\, M_S^*k_{\mathcal E}(\cdot, b)p\rangle_2\\
&=\langle \Gamma k_{\mathcal E}(\cdot, a)q\, ,\,k_{\mathcal E}(\cdot, b)p\rangle_2-p^*q+p^*S(b)S(a)^*q\\
&=\langle K_S(\cdot, a)q,K_S(\cdot,b)p\rangle_{\mathcal H(S)}-qp^*+S(b)S(a)^*q.\\
\end{split}
\]
Furthermore,

\[
\begin{split}
\langle (K_S(\cdot, a)-K_S(\cdot, 0))q\,,\, (K_S(\cdot, b)-K_S(\cdot, 0))p\rangle_{\mathcal H(S)}\\
&\hspace{-6cm}+p^*(S(b)-S(0))(S(a)^*-S(0)^*)p\\
&\hspace{-7cm}=\langle K_S(\cdot, a)q,K_S(\cdot,b)p\rangle_{\mathcal H(S)}-p^*K_S(0,a)q-p^*K_S(b,0)+p^*K_S(0,0)q+\\
&\hspace{-6cm}+p^*(S(b)-S(0))(S(a)^*-S(0)^*)p\\
&\hspace{-7cm}=\langle K_S(\cdot, a)q,K_S(\cdot,b)p\rangle_{\mathcal H(S)}-p^*(I-S(0)S(a)^*)q-p^*(I-S(b)S(0)^*)q+\\
&\hspace{-6cm}+p^*(I-S(0)S(0)^*)q+p^*(S(b)-S(0))(S(a)^*-S(0)^*)q\\
&\hspace{-7cm}=0.
\end{split}
\]

STEP 2: {\sl Case $p=q=0$.}\smallskip

We now need to check that
\[
\langle K_S(\cdot, 0)u\, ,\, K_S(\cdot, 0)v\rangle_{\mathcal H(S)}+v^*S(0)S(0)^*u=v^*u,
\]
but this is straightforward.\\

STEP 3: {\sl Mixed terms.}\smallskip

By symmetry it is enough to consider the case where $p$ and $u$ appears. We need to verify that:
\[
\left\langle K_S(\cdot, 0)u    \, ,\, (K_S(\cdot, b)-K_S(\cdot, 0))p\right\rangle+\langle S(0)^*u\, ,\, (S(b)^*-S(0)^*)p\rangle=0,
\]
but this is equivalent to
\[
K_S(b,0)-K_S(0,0)-S(0)S(0)^*+S(b)S(0)^*=0,
\]
i.e.
\[
I_r-S(b)S(0)^*-(I_r-S(0)S(0)^*)-S(0)S(0)^*+S(b)S(0)^*=0,
\]
which clearly holds.
\end{proof}

We now compute the adjoint of the above isometric operator.

We write
\begin{equation}\label{hatV}
\widehat{V}=\begin{pmatrix}\widehat{T}&\widehat{G}\\\widehat{F}&\widehat{H}\end{pmatrix}.
\end{equation}

\begin{thm}
$\widehat{V}$ is densely defined, extends to an everywhere defined isometry and its adjoint is given by
\begin{eqnarray}
\widehat{T}^*&=&R_0.
\end{eqnarray}
\end{thm}

\begin{proof} By definition of the operator range inner product we have:
\[
\begin{split}
\langle\sqrt{\Gamma}f,\Gamma  w_ap\rangle_{\mathcal H(S)}&=\langle\sqrt{\Gamma}f, M_{P_1}^*k_{\mathcal E}(\cdot, a)p\rangle_2\\
&=u^*(M_{P_1}\odot\sqrt{\Gamma}f)(a)
\end{split}
\]
and so $f=0$ if the above vanishes for all $u$ and $a$ since $M_{P_1}$ is an isometry.\\

Set $\widehat{T}^*(\sqrt{\Gamma}f)=\sqrt{\Gamma}g$. We have on the one hand
\[
\begin{split}
\langle \widehat{T}^*F, \Gamma(M_{P_1}^*k_{\mathcal E}(\cdot, a)u)\rangle_{\mathcal H(S)}&=\langle \sqrt{\Gamma}g\, ,\, \Gamma(M_{P_1}^*k_{\mathcal E}(\cdot, a)u)\rangle_{\mathcal H(S)}\\
&=\langle \sqrt{\Gamma}g,M_{P_1}^*k_{\mathcal E}(\cdot, a)u\rangle_2\\
&=u^*\left(P_1\odot \sqrt{\Gamma}g    \right)(a).
\end{split}
\]
On the other hand,

\[
\begin{split}
\langle \widehat{T}^*F, \Gamma(M_{P_1}^*k_{\mathcal E}(\cdot, a)u)\rangle_{\mathcal H(S)}&=\langle \sqrt{\Gamma}g\, ,\, \left(K_S(\cdot, a)-K(S(\cdot, 0)\right)u\rangle_{\mathcal H(S)}\\
&=u^*(\sqrt{\Gamma}g  )(a)-\sqrt{\Gamma}g  )(0))
\end{split}
\]
Hence
\[
(P_1\odot\widehat{T}^*(F))(b)=F(b)-F(0).
\]
\end{proof}

\begin{prop}
Let
\begin{equation}
\label{like-ag1}
\widehat{V}^*=\begin{pmatrix}A&B\\ C&D\end{pmatrix},
\end{equation}
where $V$ is as in \eqref{hatV}.
Then, for $f=\sum_{n=0}^\infty P_nf_n\in\mathcal H(S)$
\begin{equation}
\label{canf}
f_n=CA^nf,\quad n=0,1,\ldots
\end{equation}
and
\begin{equation}
S(x)=D+\sum_{n=1}^\infty  P_n(x)CA^{n-1}B.
\label{notequal}
\end{equation}
\end{prop}

\begin{proof}
Let $f(x)=\sum_{n=0}^\infty P_n(x)f_n\in\mathcal H(S)$. We have
\[
R_0^mf=\sum_{n=0}^\infty P_mf_{m+n},\quad m=0,1,\ldots
\]
and so $CR_0^mf=f_m$ for $m=0,1,\ldots$. To prove \eqref{notequal} we write
\[
S=S(0)+P_1\odot R_0S.
\]
Writing $S=\sum_{n=0}^\infty P_nS_n$ we conclude by applying \eqref{canf} to $R_0Su$ for $u\in\mathbb H^s$.
\end{proof}

\begin{rem}
{\rm A deep difference with respect to the classical case is that the kernel functions of $\mathbf H_2(\mathcal E)$ are not eigenvectors for $R_0$.}
\end{rem}
Note that \eqref{notequal} is not equal to $D+C\odot (I-P_1A)^{-\odot}\odot B$.
But, for $x_1=x_2=x_3=0$ we have
\[
S(x_0)=D+3x_0C(I-3x_0A)^{-1}B.
\]

\begin{rem}{\rm
Following linear system theory (see \cite{bgk1,Fuhrmann,MR0255260,MR0325201}) we will call the representation \eqref{notequal} a realization of $S$. The associated matrix $\widehat{V}$ will be called the realization matrix or the Rosenbrock
matrix. The case where $\widehat{V}^*$ is a matrix can be seen as the definition of rational functions. Unfortunately, the $CK$-product of two such functions will not be rational
in this sense. The next section deals with an important example of rational functions.}
\label{realreali}
\end{rem}

\subsection{Blaschke functions}
Equation \eqref{notequal} allows us to give a family of Schur multipliers, which we call Blaschke functions, namely those corresponding to the operator-matrix \eqref{like-ag1} to be a unitary matrix.
The definition then extends the classical case, also in the matrix-valued and possibly indefinite case; see \cite{ag} for the latter. In general there will not be $\odot$-multiplicative
decompositions of such a Blaschke ``product'' into elementary factors, hence the term {\sl function} rather than {\sl product}.

\begin{prop}
Let $B$ be $\mathbb H^{r\times r}$-valued Blaschke function, with corresponding realization matrix $\widehat{V}^*\in\mathbb H^{(N+r)\times (N+r)}$.
The corresponding multiplication operator $M_B$ is an isometry from $(\mathbf H_2(\mathcal E))^r$ into itself and the corresponding space $\mathcal H(B)$ is
finite dimensional.
\label{propo-iso}
\end{prop}

\begin{proof}
We first remark that the assumed unitarity is equivalent to the equations
\begin{align}
\label{eq1}
\mathsf A^*\mathsf A+\mathsf C^*\mathsf C&=I_N,\\
\label{eq2}
\mathsf B^*\mathsf B+\mathsf D^*\mathsf D&=I_r,\\
\mathsf D^*\mathsf C+\mathsf B^*\mathsf A&=0.
\label{offdiag}
\end{align}

Let now  $n_0,m_0\in\mathbb N_0$, and $u,v\in\mathbb H^r$. We have, with $b_n$ given by \eqref{canf},
\begin{equation}
\label{sntoeplitz}
  b_n=\begin{cases}\, \mathsf D,\,\,\hspace{11.1mm}{\rm if}\,\, n=0,\\
  \mathsf {CA^{n-1}B},\,\, {\rm if}\,\, n>0,
  \end{cases}
\end{equation}
and so
\[
\begin{split}
\langle M_BP_{n_0}u, M_BP_{m_0}v\rangle_2&=\langle P_{n_0}\odot Bu,P_{m_0}\odot Bv\rangle_2\\
&=\sum_{n,m=0}^\infty\langle   P_{n_0+n}b_nu,P_{m_0+m}b_mv\rangle_2\\
&=\sum_{\substack{n,m=0\\ n_0+n=m_0+m}}^\infty v^*b_m^*b_nu .
\end{split}
\]
We now compute
\[
\sum_{\substack{n,m=0\\ n_0+n=m_0+m}}^\infty b_m^*b_n
\]
  taking into account that the matrix $\widehat{V}$ is unitary.
When $n_0=m_0$ we have
\[
\begin{split}
\sum_{n=0}^\infty b_n^*b_n&=\mathsf{D^*D}+\sum_{n=0}^\infty \mathsf{B}^*\mathsf A^{*n}\mathsf C^*\mathsf C\mathsf A^n\mathsf{B}\\
&= \mathsf{D^*D}+\sum_{n=0}^\infty \mathsf{B}^*\mathsf A^{*n}(I_N-\mathsf A^*\mathsf A)\mathsf A^n\mathsf B\\
&=  \mathsf D^*\mathsf D+\sum_{n=0}^\infty \mathsf B^*(\mathsf A^{*n}\mathsf A^n-\mathsf A^{*(n+1)}\mathsf A^{n+1})\mathsf B\\
&=\mathsf D^*\mathsf D+\mathsf B^*\mathsf B\\
&=I_r,
\end{split}
\]
where we first used \eqref{eq1} and then \eqref{eq2}.\smallskip

Assume now $n_0<m_0$ (the case $n_0>m_0$ is obtained by symmetry). We can write:
\[
\begin{split}
\sum_{\substack{n,m=0\\ n_0+n=m_0+m}}^\infty b_m^*b_n&=\sum_{m=0}^\infty b_m^*b_{m_0-n_0+m}\\
&=\mathsf D^*\mathsf C\mathsf A^{m_0-n_0}\mathsf B+\sum_{m=1}^\infty \mathsf B^*\mathsf A^{*m}\mathsf C^*\mathsf C\mathsf A^{m}\mathsf A^{m_0-n_0}\mathsf B\\
&=\mathsf D^*\mathsf C\mathsf A^{m_0-n_0}\mathsf B+\sum_{m=1}^\infty \mathsf B^*\mathsf A^{*m}(I_N-\mathsf A^*\mathsf A)\mathsf A^{m}\mathsf A^{m_0-n_0}\mathsf B\\
&=\mathsf D^*\mathsf C\mathsf A^{m_0-n_0}\mathsf B+\mathsf B^*\mathsf A\mathsf A^{m_0-n_0}\mathsf B\\
&=(\mathsf D^*\mathsf C+\mathsf B^*\mathsf A)\mathsf A^{m_0-n_0}\mathsf B\\
&=0
\end{split}
\]
in view of \eqref{offdiag}.\smallskip

We thus have an isometry on the linear span of $P_0,P_1,\ldots$ and on the whole of $(\mathbf H_2(\mathcal E))^r$ by continuity.\smallskip

To show the finite dimensionality of the space we restrict $x=x_0,y=y_0\in(-1/3,1/3)$. We then have
\[
K_S(3x_0,3y_0)=\frac{I_r-S(x_0)S(y_0)^*}{1-9x_0y_0}=\mathsf C(I_N-3x_0\mathsf A)^{-1}(I_N-3y_0\mathsf A)^{-1}\mathsf C^*.
\]
It follows that the linear span of the functions $x_0\mapsto K_S(3x_0,3y_0)h$ ($h\in\mathbb H^r$ and $y_0\in(-1/3,1.3)$ is finite dimensional. By the uniqueness of the
axially hyperholomorphic extension the linear span of the functions $x\mapsto K_S(x,3y_0)h$ is finite dimensional We claim that they span $\mathcal H(S)$. Indeed, a function
$f\in\mathcal H(S)$ orthogonal to these functions would satisfy $f(3y_0)=0$ for $y_0\in(-1/3,1/3)$. Since $f$ is of axial type we have that $f\equiv 0$.
\end{proof}

The above computations show equivalently that:

\begin{cor}
Let $\widehat{V}^*$ given by \eqref{like-ag1}. Then the corresponding Toeplitz operator defined by the sequence \eqref{sntoeplitz} is unitary from $\ell_2(\mathbb N_0,\mathbb H^r)$ into
itself.
\end{cor}

More generally, one can take $\widehat{V}^*$ to be $\mathbb C^{(N+s)\times (N+s)}$-valued and co-isometric. Then \eqref{eq1}-\eqref{offdiag} still hold and the same proof as above  leads to:

\begin{thm}
Let $S(x)=\mathsf D+\sum_{n=1}^\infty P_n(x)\mathsf C\mathsf A^{n-1}\mathsf B$, where the realization matrix \eqref{like-ag1} is coisometric. Then the corresponding operator $M_S$ is an isometry from
$(\mathbf H_2(\mathcal E))^s$ into $(\mathbf H_2(\mathcal E))^r$.
\end{thm}

\begin{rem}
  \label{remark-ad3}
  {\rm Take $S_1,\ldots, ,S_N$ to be $N$ Schur multipliers (say, $\mathbb H$-valued) with associated finite dimensional $\mathcal H(S_j)$ spaces, $j=1,\ldots, N$, and let $t_1,\ldots, t_N$ to be real
    numbers such that $\sum_{n=1}^Nt_n^2=1$. The function
    \[
S(x)=\begin{pmatrix}t_1S_1(x)&\cdots &t_NS_N(x)\end{pmatrix}
     \]
     is a Schur multiplier from $\left(\mathbf H_2(\mathcal E)\right)^N$ into $\mathbf H_2(\mathcal E)$, and the associated reproducing kernel space is finite dimensional, but will
     not be isometrically included in $\mathbf H_2(\mathcal E)$ when $N>1$; see \cite[p. 71]{ad3} for the complex setting. The argument is the same here.}
 \end{rem}

\begin{rem}{\rm In fact the results in the present section still hold when $\widehat{V}^*$ is isometric, but not necessarily a matrix. Then the corresponding multiplier is called {\sl inner}.
    The study of these multipliers will be considered elsewhere. Similarly, one could assume unitarity with respect to an indefinite metric. This aspect of the theory will also
  be treated in a separate publication.}
\end{rem}

\subsection{Rational functions}
We now define rational functions in the present setting. We first recall that any $\mathbb C^{n\times m}$-valued rational function, say $M(z)$, with no pole at the origin can be written in the form
\begin{equation}
  \label{real2}
M(z)=H+zG(I-zT)^{-1}F,
\end{equation}
where $H,G,T,F$ are matrices of appropriate sizes. Expression \eqref{real2} is called a realization (centered at the origin). See Remark \ref{realreali} above for references. We also recall
the formulas for the product and inverse of rational functions. Note that, since we consider possibly non-square functions, the sum will be a special case of the product since
\[
  \begin{pmatrix}M_1(z)& I_n\end{pmatrix}\begin{pmatrix} I_m\\ M_2(z)\end{pmatrix}=M_1(z)+M_2(z)
\]
where $M_1$ and $M_2$ are $\mathbb C^{n\times m}$-valued.\\

Assuming in \eqref{real2} that $n=m$ and $H$ invertible, one has the formula
\begin{equation}
\label{inv12345}
M(z)^{-1}=H^{-1}-zH^{-1}G(I-zT^\times)^{-1}FH^{-1},
\end{equation}
where $T^\times=T-GH^{-1}F$,
and with $M_j(z)=H_j+zG_j(I-zT_j)^{-1}F_j$, $j=1,2$, two rational functions of compatible sizes,
\[
M_1(z)M_2(z)=H+zG(I-zT)^{-1}F,
\]
with $H=H_1H_2$ and
\begin{equation}
  \label{prod2}
  T=\begin{pmatrix}  T_1&G_1T_2\\ 0&T_2\end{pmatrix},\quad G=\begin{pmatrix} G_1 \\ H_1G_2\end{pmatrix},\quad F=\begin{pmatrix}F_1H_2&F_2\end{pmatrix}.
  \end{equation}
\begin{defn}
  The $\mathbb H^{r\times s}$-valued function $R(x)$ hyperholomorphic of axial type is called rational if its restriction to the real axis is a rational function of the real variable,
  with quaternionic coefficients.
\end{defn}

\begin{thm}
  The $\mathbb H^{r\times s}$-valued function $R(x)$ hyperholomorphic of axial type, and defined at the origin, is rational if and only if $x_0\mapsto R(3x_0)$ can be written as
  \begin{equation}
    R(3x_0)=\mathsf D+3x_0\mathsf C(I-3x_0\mathsf A)^{-1}\mathsf B
  \end{equation}
  where $\mathsf{ A,B,C,B}$ are quaternionic matrices of appropriate sizes.
  \end{thm}

  Equivalently:

  \begin{thm}
  The $\mathbb H^{r\times s}$-valued function $R(x)$ hyperholomorphic of axial type, and defined at the origin, is rational if and only if it can be written as
  \begin{equation}
    R(x)=\mathsf D+\sum_{n=0}^\infty P_n(x)\mathsf C\mathsf A^{n-1}\mathsf B
  \end{equation}
  where $\mathsf{ A,B,C,B}$ are quaternionic matrices of appropriate sizes.
\end{thm}

We will not give the proofs of the above results, which follow easily from the previous analysis in the paper. One still has the formulas \eqref{inv12345} and \eqref{prod2} for a real variable $x_0$,
but not for the $CK$-product. So the product of rational hyperholomorphic functions of axial type is not compatible with the $CK$ product. To emphasize this point, we now make the connection
with rational functions of the Fueter variables, as studied in \cite{MR2275397,assv,MR2124899} (see also \cite{MR93c:30059}, and see \cite{MR3819695} for the split quaternionic case).
There, rational functions are characterized by the formula (we do not specify the sizes of the various quaternionic matrices)

\begin{equation}
\label{realR}
\begin{split}
R(x)&=\mathsf D+\mathsf C\odot (I-(\zeta_1\mathsf A_1  +\zeta_2\mathsf A_2+\zeta_3\mathsf A_3)^{-\odot}\odot(\zeta_1\mathsf B_1+\zeta_2\mathsf B_2+\zeta_3\mathsf B_3)\\
&=\mathsf D+\mathsf C\odot(I-\zeta \mathsf A)^{-\odot}\odot(\zeta \mathsf B),
\end{split}
\end{equation}
with
\[
\mathsf A=\begin{pmatrix}\mathsf A_1\\ \mathsf A_2\\ \mathsf A_3\end{pmatrix}\quad{\rm and}\quad \mathsf B=\begin{pmatrix}\mathsf B_1\\ \mathsf B_2\\ \mathsf B_3\end{pmatrix},
\]
and the variable here is
\[
\zeta=\begin{pmatrix}\zeta_1&\zeta_2&\zeta_3\end{pmatrix}.
\]
We look at the special case where
\[
\mathsf  A=\begin{pmatrix}{\mathbf e_1}\mathscr A    \\ {\mathbf e_2}\mathscr A\\ {\mathbf e_3}\mathscr A\end{pmatrix}\quad{\rm and}\quad \mathsf B=\begin{pmatrix}{\mathbf e_1}\mathscr B\\
    {\mathbf e_2}\mathscr B\\ {\mathbf e_3}\mathscr B\end{pmatrix},
\]
where $\mathscr A$ and $\mathscr B$ are matrices of appropriate sizes and
with quaternionic coefficients. Since
$P_1(x)=\zeta_1(x)\mathbf e_1+\zeta_2(x)\mathbf e_2+\zeta_3(x)\mathbf e_3$ we can then rewrite \eqref{realR} as
\[
R(x)=\mathsf D+\mathsf C\odot(I-P_1(x)\mathscr A)^{-\odot}\odot P_1(x){\mathscr B},
\]
which will not be in general a series in the $P_n$, but is a series in the $\zeta^\alpha$. We can define such elements as the rational functions associated with the polynomials $P_n$. Then,
things make sense in terms of realizations, with the usual formulas, but we do not get power series in $P_n$, even when $\mathcal A$ is nilpotent.\\

\subsection{A structure theorem}
In the classical, and scalar-valued, setting, Beurling's theorem asserts that a closed subspace $\mathcal N$ of the Hardy space $\mathbf H_2(\mathbb D)$
is invariant under multiplication by the complex variable if and only if it is of the form $\mathcal N=j\mathbf H_2(\mathbb D)$, where $j$ is an inner function. The space
$\mathcal M=\mathbf H_2(\mathbb D)\ominus\mathcal N$ is then isometrically included in $\mathbf H_2(\mathbb D)$, backward-shift invariant, and has reproducing kernel equal to
$\frac{1-j(z)\overline{j(w)}}{1-z\overline{w}}$. One motivation for the theory of de Branges-Rovnyak is to characterize reproducing kernel Hilbert spaces of functions which are contractively
included in $\mathbf H_2(\mathbb D)$, and $R_0$-invariant. Rather than the latter, one assumes that
the inequality
\begin{equation}
  \label{ineq12345}
\|R_0f\|_{\mathcal M}^2\le \|f\|_{\mathcal M}^2-|f(0)|^2
\end{equation}
holds. This inequality implies contractive inclusion in the Hardy space,
\begin{equation}
f\in\mathcal M\,\,\Longrightarrow \,\, f\in \mathbf H_2(\mathbb D)\quad {\rm and}\quad\|f\|_{\mathbf H_2(\mathbb D)}\le \|f\|_{\mathcal M}
\end{equation}
and that in particular the space is a reproducing kernel Hilbert space since
\[
|f(w)|\le \frac{1}{\sqrt{1-|w|^2}}\|f\|_{\mathbf H_2(\mathbb D)}\le \frac{1}{\sqrt{1-|w|^2}}\|f\|_{\mathcal M},\quad w\in\mathbb D.
\]

\begin{thm}
  \label{guyker?}
  Let $\mathcal M$ be a Hilbert space of functions analytic in the open unit disk and such that \eqref{ineq12345} holds in $\mathcal M$. Then,
  $\mathcal M$ is contractively included inside $\mathbf H_2(\mathbb D)$ and there is a (possibly $\ell_2$-valued) analytic function
  $s$ such that the reproducing kernel of $\mathcal N$ is $\dfrac{1-s(z)s(w)^*}{1-z\overline{w}}$.
\end{thm}

We refer to the notes in \cite[p. 206]{adrs} for some history on Theorem \ref{guyker?}, but mention the papers \cite{MR250051}. Guyker characterized the
spaces for which the inclusion is isometric; see \cite[p. 187]{adrs}, \cite{MR1031663,MR1257108}.
For an illustration of the contractive inclusion, see Remark \ref{remark-ad3} above.\\

A general version of Theorem \ref{guyker?}, in the operator-valued and Pontryagin space case, has been proved in  \cite[Theorem 3.1.2, p. 85]{adrs}.
We will present in a subsequent paper the general version of the result, in the Pontryagin and operator-valued function case. The purpose of this section is to illustrate the power of the
methods used here on a simple case. Note that \eqref{ineq} is a weakening of \eqref{fundamental987}.

\begin{thm}
Let $\mathcal H$ be a Hilbert space of $\mathbb H^r$-valued functions axially  hyperholomorphic in $\mathcal E$, and $R_0$-invariant and satisfying
\begin{equation}
\label{ineq}
\|R_0f\|_{\mathcal H}^2  \le \|f\|_{\mathcal H}^2  -f(0)^*f(0),\quad f\in\mathcal H.
\end{equation}
Then there exist a right  quaternionic Hilbert space $\mathcal C$ and a $\mathbf L(\mathcal C,\mathbb H^r)$-valued function $S$  hyperholomorphic of axial type
such that $\mathcal H=\mathcal H(S)$.
\end{thm}

\begin{proof}
It follows from \eqref{ineq} that $\mathcal{H}$ is contractively included in $(\mathbf H_2(\mathcal E))^r$, and that the
  operator $R_0$ is bounded (another argument, still valid in the quaternionic Pontryagin space setting would be to prove that $R_0$ is closed, thanks to the
reproducing kernel property, and use the closed graph theorem; see \cite[Theorem 5.1.16, p. 74]{zbMATH06658818} for the latter). The point evaluation at the origin, which we will denote by $C$,
is also bounded since the space is contractively included in the Hardy space and its norm is larger than the Hardy space norm. We can thus rewrite \eqref{ineq} as
\[
R_0^*R_0+C^*C\le I.
\]
Since the adjoint of a contraction between Hilbert space is still a contraction we have
\[
I_{{\mathcal H}\oplus\mathbb H^r}-\begin{pmatrix} R_0\\ C\end{pmatrix}\begin{pmatrix} R_0\\ C\end{pmatrix}^*\ge 0,
\]
and we can factorize the quaternionic positive operator via a right  quaternionic Hilbert space $\mathcal C$ as
\[
I_{{\mathcal H}\oplus\mathbb H^r}-\begin{pmatrix} R_0\\ C\end{pmatrix}\begin{pmatrix} R_0\\ C\end{pmatrix}^*=\begin{pmatrix} B\\ D\end{pmatrix}\begin{pmatrix} B\\ D\end{pmatrix}^*,
\]
with
\[
\begin{pmatrix} B\\ D\end{pmatrix}\in\mathbf L(\mathcal C,\mathcal H\oplus \mathbb H^r).
\]
The operator matrix
\[
\begin{pmatrix}R_0&B\\ C&D\end{pmatrix}\,\,:\,\, \mathcal H\oplus\mathcal C\,\,\longrightarrow\,\,\mathcal H\oplus\mathbb H^r
\]
is co-isometric, and the $\mathbf L(\mathcal C,\mathbb H^r)$-valued function $S$  defined by
\[
S(x)=D+\sum_{n=1}^\infty P^{\odot n}(x)CR_0^nB
\]
is a Schur multiplier. To conclude we show that $\mathcal H=\mathcal H(S)$.
From the definition of $S$ we have for $x_0,y_0\in(-1/3,1/3)$
\[
K_S(3x_0,3y_0)=C(I-3x_0R_0)^{-1}(I-3y_0R_0)^{-*}C^*.
\]
We have
\[
K_S(3x_0,3y_0)-3x_0K_S(3x_0,3y_0)3y_0=I_r-S(3x_0)S(3y_0)^*
\]
and the result follows from axially hyperholomorphic extension on the left with respect to $x$ and on the right with respect to $y$.
\end{proof}

\section{Spaces $\mathcal L(\Phi)$}
\setcounter{equation}{0}
\label{herg}
We now consider the counterpart of $\mathcal L(\Phi)$ spaces, see \eqref{lphi}, in the present setting, and first briefly review the classical case.
Functions analytic in the open unit disk and with a positive real part there will be called here Herglotz functions (they are called
Carath\'eodory functions in Akhiezer's book \cite[p. 116]{akhiezer_russian}). An Herglotz function, say $\Phi$, is characterized by an integral representation of the form
\begin{equation}
\Phi(z)=im+\int_{[0,2\pi]}\frac{e^{it}+z}{e^{it}-z}d\sigma(t)
\label{herg}
\end{equation}
where $m\in\mathbb R$,  $\sigma$ is an increasing function, and \eqref{herg} is a Stieltjes integral. They play an important role in the trigonometric moment problem, operators models for
isometric and unitary operators in Hilbert spaces and in the theory of dissipative discrete systems, and have been extended to various more general frameworks; see e.g.
\cite{MR330390911,MR2336046,MR3043592,MR3350215,kl1,MR2182589,zbMATH06823280}.\\

In a way similar to Theorem \ref{dono}, consider a function $\Phi$ defined in a real neighborhood $(-\epsilon, \epsilon)$ of the origin and such that the kernel
\[
\frac{\Phi(a)+\overline{\Phi(b)}}{2(1-ab)}
\]
is positive definite there. Then it is the restriction to $(-\epsilon, \epsilon)$ of a uniquely defined Herglotz function, and the corresponding kernel

\[
  L_\Phi(z,w)=\frac{\Phi(z)+\overline{\Phi(w)}}{2(1-z\overline{w})}
\]
is positive definite in the open unit disk. The factor $2$ in the kernels is to get nicer realization formulas (such as \eqref{pv}) and follows basically from the Cayley transform
$z\mapsto \dfrac{1-z}{1+z}$, which maps Herglotz functions into Schur functions.\smallskip

The corresponding reproducing kernel Hilbert space and its applications to operator models was first characterized and studied by de Branges; see \cite{dbbook,dbs}.
It is important to note that a Herglotz function need not be bounded, and hence need not be a multiplier of the Hardy space.\\

Using the reproducing kernel space $\mathcal L(\Phi)$ associated with $L_\Phi$ (or directly from \eqref{herg}), one can characterize Herglotz functions in terms of a realization of the form
\begin{equation}
\label{pv}
\Phi(z)=ia+C(I+zV)  (I-zV)^{-1}C^*,
\end{equation}
where $V$ is coisometric in some Hilbert space, and $C^*$ is a continuous map from the coefficient space (the complex numbers when the functions are scalar) into that Hilbert space. Note that
\eqref{pv} can be rewritten as
\begin{equation}
\label{phivc}
\Phi(z)=ia+CC^*+2\sum_{n=1}^\infty z^nCV^nC^*,\quad z\in\mathbb D.
\end{equation}

In this section we study the counterpart of the kernel $L_\Phi$ in our setting, and give the counterpart of the expansion \eqref{phivc}, and study connections with Toeplitz operators. In the
following definition (and also in Section \ref{carasec} below) we use the term {\sl multiplier} although the operator of $CK$-multiplication by the given function need not be bounded in the
Hardy space.

\begin{defn}
  A $\mathbb H^{r\times r}$-valued axially hyperholomorphic function $\Phi$
    is called a Herglotz multiplier if  the kernel
\begin{equation}
L_\Phi(x,y)=\frac{1}{2}\sum_{n=0}^\infty(P_n\odot \Phi)(x)P_n(y)^*+P_n(x)((P_n\odot\Phi  )(y))^*
\label{kphi123}
\end{equation}
is positive definite in $\mathcal E$.
\end{defn}

When the operator of $CK$-multiplication by $\Phi$ is bounded in the Hardy space $(\mathbf H_2(\mathcal E))^r$, we can replace \eqref{kphi123} by the condition (see Remark \ref{new})
\begin{equation}
\Gamma\stackrel{\rm def.}{=}\frac{M_\Phi+M_\Phi^*}{2}\ge 0.
\end{equation}
\begin{prop}
We note the following property of $\Gamma$:
\begin{equation}
\label{mphi}
M_{P_1}\Gamma M_{P_1}^*=\frac{1}{2}\left\{M_\Phi(I-C^*C)+(I-C^*C)M_\Phi^*\right\}.
\end{equation}
\end{prop}

\begin{proof}
Using the fundamental equality \eqref{fundamental987}, we can write
\[
\begin{split}
M_{P_1}\Gamma M_{P_1}^*&=M_{P_1}\left(\frac{M_\Phi+M_\Phi^*}{2}\right)M_{P_1}^*\\
&=\frac{1}{2}M_\Phi M_{P_1}M_{P_1}^*+\frac{1}{2}  M_{P_1}M_{P_1}^*M_\Phi^*\\
&=\frac{1}{2}M_{\Phi}(I-C^*C)+\frac{1}{2}(I-C^*C)M_\Phi^*.
\end{split}
\]
\end{proof}

We make now some remarks on the above kernel. It seems difficult to find a direct counterpart of \eqref{pv} (in view of \eqref{obstruct}), let alone of \eqref{herg}. As for the case of Schur
multiplier, we look for a characterization of the coefficients of $\Phi$ in its expansion along the $P_n$. When $x=x_0$ and $y=y_0$ belong to $(-1/3,1/3)$ the kernel $L_\Phi$ becomes
\begin{equation}
\label{lphiR}
L_\Phi(3x_0,3y_0)=\frac{\Phi(3x_0)+\Phi(3y_0)^*}{1-9x_0y_0}.
\end{equation}
As for the case of Schur multipliers, this restriction is enough to get back the kernel $L_\Phi$ in view of the axial symmetry of the functions.
\begin{defn}
We denote by $\mathcal L(\Phi)$ the reproducing
kernel  Hilbert space of axially hyperholomorphic functions with reproducing kernel \eqref{kphi123}.
\end{defn}

The Hilbert space $\mathcal L(\Phi)$ is the completion of the linear span of functions of the form $L_{\Phi}(x,y)p$, $p\in\mathbb H^r$, thus it consists of hyperholomorphic functions of axial type.

\begin{thm}
  Let $\Phi(x)=\sum_{n=0}^\infty P_n(x)\Phi_n$ be an axially hyperholomophic function in $\mathcal E$. Then $\Phi$ is a Herglotz multiplier if and only if the coefficients $\Phi_n$ can be written as
\begin{equation}
  \label{phin}
  \Phi_n=\begin{cases}\,\, CC^*,\,\,\hspace{11.4mm} n=0,\\
    \,\,2CV^{*n}C^*, \quad n=1,2,\ldots\end{cases}
\end{equation}
where $V$ is an isometry in a Hilbert space, say $\mathfrak H$, and $C$ is a continuous map from $\mathfrak H$ into $\mathbb H^r$.\smallskip
\end{thm}

\begin{proof}
We first prove the sufficiency. We have for $a\in(-1,1)$
\[
\Phi(a)=CC^*+2\sum_{n=1}^\infty a^nCV^{*n}C^*=C(I+aV^*)(I-aV^*)^{-1}C^*.
\]
Thus, for $a,b\in (-1,1)$,
\[
\frac{\Phi(a)+\Phi(b)^*}{2(1-ab)}=C(I-aV^*)^{-1}(I-bV^*)^{-*}C^*
\]
and thus the corresponding function \eqref{lphiR} is positive definite in $(-1/3,1/3)$, and we get the result by hyperholomorphic extension on the left with respect to $x$ and on
the right with respect to $y$ since the functions are assumed axially hyperholomorphic.\\

We now turn to the proof of the direct statement, and divide the proof in a number of steps. We follow \cite[Proof of Theorem 5.2, p. 708]{atv1}.
We write $P_1(3y_0)$ rather than $3y_0$ to emphasize the axially symmetric hypercomplex extension to be used. At the end of section alternative steps 1 and 2 are given when $M_\Phi$ is bounded.\\

STEP 1: {\sl The linear relation (see Definition \ref{real1!!!}) of $\mathcal L(\Phi)\times\mathcal L(\Phi)$ defined by the span of the functions
  \begin{equation}
\label{rela}
\left(L_\Phi(\cdot, 3y_0)\overline{P_1(3y_0)}q,(L_\Phi(\cdot, 3y_0)-L_\Phi(\cdot,0))q\right),\quad y_0\in(-1/3,1/3),\quad q\in\mathbb H^r,
    \end{equation}
is densely defined and isometric.}\smallskip

Let $F$ be orthogonal to the domain of the relation, then $P_1(3y_0)F(3y_0)\equiv 0$ and so $F$
is identically equal to $0$ by axially hypercomplex extension.\smallskip

Next, we need to show that, for  $x_0,y_0,\in(-1/3,1/3)$ and $p,q\in\mathbb H^r$, we have:
\[
\begin{split}
\langle L_\Phi(\cdot, 3y_0)\overline{P_1(3y_0)}q, L_\Phi(\cdot, 3x_0)&\overline{P_1(3x_0)}p\rangle_{\mathcal L(\Phi)}\\
&=\langle (L_\Phi(\cdot, 3y_0)-L_\Phi(\cdot,0))q,(L_\Phi(\cdot, 3x_0)\\&  \hspace{3cm}-L_\Phi(\cdot,0))p\rangle_{\mathcal L(\Phi)}.
\end{split}
\]
This amounts to check that
\[
P_1(3x_0)L_\Phi(3x_0,3y_0)\overline{P_1(3y_0)}=L_\Phi(3x_0,3y_0)-L_\Phi(3x_0,0)-L_\Phi(0,3y_0)+L_\Phi(0,0)
\]
but this is a direct consequence of the definition of $L_\Phi$ since
\[
L_\Phi(3x_0,0)+L_\Phi(0,3y_0)-L_\Phi(0,0)=\dfrac{\Phi(3x_0)+\Phi(3y_0)^*}{2}
\]
and
\[
L_\Phi(3x_0,3y_0)-P_1(3x_0)L_\Phi(3x_0,3y_0) \overline{P_1(3y_0)}=\dfrac{\Phi(3x_0)+\Phi(3y_0)^*}{2}.
\]
There is thus an everywhere defined isometric operator such that, for $q\in\mathbb{H}^r$
\[
  V\left(L_\Phi(\cdot, 3y_0)\overline{P_1(3y_0)}q\right)=(L_\Phi(\cdot, 3y_0)-L_\Phi(\cdot,0))q,\quad y_0\in(-1/3,1/3).
\]

STEP 2: {\sl We have
\[
V^*=R_0.
\]
}
On the one hand
\[
  \begin{split}
    \langle V^*F, (L_\Phi(\cdot, 3y_0)\overline{P_1(3y_0)})q)\rangle_{\mathcal L(\Phi)}&=\langle F,(L(\cdot, 3y_0)-L(\cdot,0))q\rangle_{\mathcal L(\Phi)}\\
    &=q^*(F(3y_0)-F(0))
  \end{split}
  \]
  and on the other hand
  \[
    \begin{split}
      \langle V^*F, (L_\Phi(\cdot, 3y_0)\overline{P_1(3y_0)}q)      \rangle_{\mathcal L(\Phi)}&=\langle F,(L_\Phi(\cdot, 3y_0)\overline{P_1(3y_0)})q)\rangle_{\mathcal L(\Phi)}\\
      &=q^*P_1(3y_0)(V^*F)(3y_0).
    \end{split}
  \]
  Hence,
  \[
P_1(3y_0) (V^*F)(3y_0)=F(3y_0)-F(0),
   \]
  and hence the result by axially hypercomplex extension.\\

STEP 3: {\sl Let $C$ denote the evaluation at the origin in $\mathcal L(\Phi)$. It holds that}
\[
(C^*q)(x)=L_\Phi(x,0)q=(\Phi(x)+\Phi(0)^*)q,\quad q\in\mathbb H^r.
\]

We have, with $p,q\in\mathbb H^r$,
\[
  \begin{split}
    \langle C^*q\,,\, L_\Phi(\cdot, x)p\rangle_{\mathcal L(\Phi)}=\langle q\,,\, L(0,x)p\rangle_{\mathbb H^r}=\frac{1}{2}((\Phi(0)+\Phi(x)^*)p)^*q,
  \end{split}
  \]
  and hence  the result.\\

STEP 4: {\sl We prove \eqref{phin}.}\\

With $q\in\mathbb H^r$ we write
\[
\begin{split}
\Phi(x)q&=(\Phi(x)+\Phi(0)^*)q-\Phi(0)^*q\\
&=2(C^*q)(x)-\Phi(0)^*q\\
&=2\sum_{n=0}^\infty CV^{*n}C^*q-\Phi(0)^*q\\
&=CC^*+2\sum_{n=1}^\infty CV^{*n}C^*q+CC^*-\Phi(0)^*q\\
&=CC^*+2\sum_{n=1}^\infty CV^{*n}C^*q+\frac{\Phi(0)-\Phi(0)^*}{2}q
\end{split}
\]
since $CC^*=\dfrac{\Phi(0)+\Phi(0)^*}{2}$.
\end{proof}

\begin{cor}
Let $\Phi$ be a Herglotz multiplier. The space $\mathcal L(\Phi)$ is $R_0$ invariant.
\end{cor}

As for Schur multiplier one has:

\begin{thm}
A space $\mathcal L(\Phi)$ is finite dimensional if and only if the operator $V$ can be chosen to act in a finite dimensional space (and hence is unitary).
\end{thm}


In terms of Toeplitz matrices we have:

\begin{thm}
$\Phi$ is a $\mathbb H^{r\times r}$-valued bounded Herglotz multiplier if and only if the infinite block matrix
$\left(\Phi_{n-m}\right)_{n,m=0}^\infty$ (with $\Phi_{-m}=\Phi_m^*$) defines a bounded positive operator.
\end{thm}

\begin{proof}
Let $\Phi$ be such that the kernel $L_\Phi$ is positive definite in $\mathcal E$ (at this stage we do not assume yet that the associated operator $M_\Phi$ is bounded).
Then, setting $x_1=x_2=x_3=0$ and applying the map $\chi$ we see that the kernel
\[
\frac{\chi(\Phi)(3x_0)+\left(\chi(\Phi)(3y_0)\right)^*}{1-9x_0y_0}
\]
is positive definite in $(-1/3,1/3)$. Setting $a=3x_0$ and $b=3y_0$, and with $\varphi(a)=\chi(\Phi)(3x_0)$ we see that the kernel
\[
\frac{\varphi(a)+\varphi(b)^*}{1-ab}
\]
is positive definite on $(-1,1)$. Thus it extends to a positive definite function to the open unit disk, and the extended $\varphi$ has a positive real part there.
The block Toeplitz matrices
\[
\left(\chi(\Phi_{n-m})\right)_{n,m=0}^N,
\]
with $\Phi_{-m}=\Phi_m^*$, are thus non-negative. If $M_\Phi$ is bounded, the infinite Toeplitz matrix  $\left(\chi(\Phi_{n-m})\right)_{n,m=0}^\infty$ defines a bounded positive operator, and so does
$\left(\Phi_{n-m}\right)_{n,m=0}^\infty$ by Proposition \ref{2-9-0}.\\

 Conversely, if $\left(\Phi_{n-m}\right)_{n,m=0}^\infty$ defines a bounded positive operator, the function
\[
\begin{pmatrix}I_r& P_1(x)I_r&\cdots\end{pmatrix}\left(\Phi_{n-m}\right)_{n,m=0}^\infty\begin{pmatrix}I_r\\ I_r\overline{P_1(y)} \\ \vdots\end{pmatrix}
\]
is positive definite. It can be rewritten as the kernel $L_\Phi$. Indeed
\[
\begin{split}
\begin{pmatrix}I_r& P_1(x)I_r&\cdots\end{pmatrix}\left(\Phi_{n-m}\right)_{n,m=0}^\infty\begin{pmatrix}I_r\\ I_r\overline{P_1(y)} \\ \vdots\end{pmatrix}&=\sum_{n,m=0}^\infty P_n(x)\Phi_{n-m}\overline{P_m(y)}
\end{split}
\]
while
\[
\begin{split}
L_\Phi(x,y)&=\sum_{n=0}^\infty \left(\sum_{m=0}^\infty P_{n+m}(x)\Phi_m\right)\overline{P_m(y)}+P_n(x)\left(\sum_{m=0}^\infty\Phi_m^*\overline{P_{m+n}(y)}\right)\\
&=\sum_{n=0}^\infty\sum_{m=0}^\infty P_{n}(x)\Phi_{n-m}\overline{P_m(y)}
\end{split}
\]
with $\Phi_{-m}=\Phi_m^*$ for $m=0,1,\ldots$.\\
\end{proof}
We conclude with a computation of the linear relation associated to $\Phi$ when $M_\Phi$ is a bounded operator. The computations are longer, but avoid axially symmetric extensions.
The relation \eqref{rela} becomes:
\begin{equation}
  \label{linearel}
  \left(\sum_{j}\Gamma(M_{P_1}^*k_{\mathcal E}(\cdot, a_j)q_j), \sum_{j}L_\Phi(\cdot, a_j)q_j  -L_\Phi(\cdot, 0)\left(\sum_{j} q_j\right)\right),
\end{equation}
 where $u_1,v_1\ldots\in\mathcal E$ and $p_1,q_1,\ldots \in
\mathbb H^r$.\smallskip

NEW STEP 1: {\sl The linear relation spanned by the elements \eqref{linearel} is isometric and  densely defined, and hence extends to the graph of an everywhere defined
  isometry.}\\

To prove this claim, we let $f=\sum_j k_{\mathcal E} (\cdot, u_j)q_j$ and $g=\sum_\ell k_{\mathcal E} (\cdot, v_\ell)p_\ell$.  Using \eqref{mphi} we can write:
\[
\begin{split}
\langle \Gamma (M_{P_1}^*f),\Gamma (M_{P_1}^*g)\rangle_{\mathcal L(\Phi)}&=\langle M_{P_1}^*f,\Gamma (M_{P_1}^*g)\rangle_2\\
&=\langle f,M_{P_1}\Gamma (M_{P_1}^*g)\rangle_2\\
&=\frac{1}{2}\langle f, (M_{\Phi}(I-C^*C)+(I-C^*C)M_\Phi^*)g\rangle_2\\
&=\langle f,\Gamma g\rangle-\frac{1}{2}\left(CM_\Phi^*g\right)^*Cf-\frac{1}{2}\left(Cg\right)^*CM_\Phi^*f\\
&=\sum_{j,\ell}p_\ell^*L_\Phi(v_\ell,u_j)q_j-\frac{1}{2}\left(\sum_\ell \Phi(v_\ell)^*p_\ell\right)^*\left(\sum_{j} q_j\right)-\\
&\hspace{5mm}-\frac{1}{2}\left(\sum_\ell p_\ell\right)^*\left(\sum_{j} \Phi(u_j)^*q_j\right)\\
&\\
&=\sum_{j,\ell}p_\ell^*L_\Phi(v_\ell,u_j)q_j-\left(\sum_\ell L_\Phi(0,v_\ell)^*p_\ell\right)^*\left(\sum_{j} q_j\right)-\\
&\hspace{5mm}-\left(\sum_\ell p_\ell\right)^*\left(\sum_{j} L_\Phi(0,u_j)^*q_j\right)\\
&\hspace{5mm}+\left(\sum_\ell p_\ell\right)^*L_\Phi(0,0)\left(\sum_{j} L_\Phi(0,u_j)^*q_j\right)\\
&\\
&\hspace{-4cm}=\left\langle \sum_{j}L_\Phi(\cdot, u_j)q_j  -L_\Phi(\cdot, 0)\left(\sum_{j} q_j\right)\, ,\, \sum_{\ell}L_\Phi(\cdot, v_\ell)p_\ell
  -L_\Phi(\cdot, 0)\left(\sum_{\ell} p_\ell\right)\right\rangle_{\mathcal L(\Phi)}.
\end{split}
\]

There is thus an everywhere defined isometric operator such that
\[
V\left(\Gamma M_{P_1}^*k_{\mathcal E}(\cdot, u)q\right)=L_\Phi(\cdot, u)q-L_\Phi(\cdot, 0)q,\quad u\in\mathcal E.
\]

NEW STEP 2: {\sl We have
\[
V^*=R_0.
\]
}

Indeed, let $V^*F=\sqrt{\Gamma}f$. On the one hand,
\[
\begin{split}
\langle V^*F,\sqrt{\Gamma}\left(M_{P_1}^*k_{\mathcal E}(\cdot, u)\right)q\rangle_{\mathcal L(\Phi)}  &=\langle \sqrt{\Gamma}f, M_{P_1}^*k_{\mathcal E}(\cdot, u)q\rangle_2\\
&=\langle M_{P_1}V^*F,k_{\mathcal E}(\cdot, u)q\rangle_2\\
&=q^*(P_1\odot V^*F)(u).
\end{split}
\]
On the other hand,
\[
\begin{split}
  \langle V^*F,\sqrt{\Gamma}\left(M_{P_1}^*k_{\mathcal E}(\cdot, u)\right)q\rangle_{\mathcal L(\Phi)}  &=\langle F, V\left(\sqrt{\Gamma}\left(M_{P_1}^*k_{\mathcal E}(\cdot, u)
    \right)q\right)\rangle_{\mathcal L(\Phi)}\\
&=\langle F, L_\Phi(\cdot, u)q-L_\Phi(\cdot, 0)q\rangle_{\mathcal L(\Phi)}\\
&=q^*(F(u)-F(0)).
\end{split}
\]
Hence, using the formulas \eqref{in11}-\eqref{in22} for the operator range inner product, we have
\[
(P_1\odot V^*F)(u)=F(u)-F(0).
\]

\section{The half-space case}
\setcounter{equation}{0}
We first recall that the classical Hardy space of the open right half-plane $\mathbb C_r$ is the reproducing kernel Hilbert space with reproducing kernel equal to $\dfrac{1}{2\pi(z+\overline{w})}$
(the factor $2\pi$ appears because of Cauchy's formula), and can be characterized as the space of power series of the form
\begin{equation}
\label{hoffman}
b(z)=\sum_{n=0}^\infty\frac{(1-z)^n}{(1+z)^{n+1}}b_n
\end{equation}
where the complex numbers $b_n$ satisfy $\sum_{n=0}^\infty |b_n|^2<\infty$; see for instance
\cite{MR1102893}. We will denote this space by $\mathbf H_2(\mathbb C_r)$. The purpose of this section is to define and begin a study of the counterpart of the space
$\mathbf H_2(\mathbb C_r)$ in the present setting. A more detailed analysis will be presented in a sequel to the present work. We give three equivalent characterizations of the new space,
respectively in terms of a reproducing kernel, restriction to the positive real axis and series expansion analogous to \eqref{hoffman}. We first define what will be the counterpart of
$\mathbb C_r$. To that purpose, consider the function $w(x)$ defined in \eqref{wx}; it has for (unique) $CK$-extension
\begin{equation}
W_1(x)=(1-P_1(x))\odot(1+P_1(x))^{-\odot}.
\end{equation}
This function is intrinsic hyperholomorphic of axial type by Remark \ref{rmk212}, in fact in a neighborhood of the origin we can write it as
\[
\begin{split}
W_1(x)&=(1-P_1(x))\odot \sum_{n=0}^{\infty} (-1)^n P_1(x)^{\odot n}=\sum_{n=0}^{\infty} (-1)^n P_1(x)^{\odot n} \\&  \hspace{3cm}+ \sum_{n=0}^{\infty} (-1)^{n+1} P_1(x)^{\odot n +1}\\
&=1+2 \sum_{n=1}^{\infty} (-1)^{n} P_1(x)^{\odot n}.
\end{split}
\]
We define $W_n(x)=W_1^{\odot n}(x)$ and we set
\[
K_W(x,y)=\sum_{n=0}^\infty W_n(x)\overline{W_n(y)}.
\]
Note that $W_n(0)=1$ and also that $W_n(x)$ is hyperholomorphic of axial type, being a finite $CK$-product of intrinsic series in $P_1(x)$. On the other hand for $x_1=x_2=x_3=0$ and $x_0>0$,
\[
|W_n(x_0)|=\left|\frac{1-3x_0}{1+3x_0}\right|<1.
\]
Using the arguments in Lemma \ref{fog} we can prove the following:
\begin{lem}\label{Wenne}
The series
\begin{equation}
\label{pi-r-1}
\sum_{n=0}^\infty |W_n(x)|^2
\end{equation}
converges in a neighborhood of $x=1/3$.
\end{lem}
\begin{proof}
The function $W_1(x)$ is hyperholomorphic by its definition. Since we have  $W_1(1/3)=0$ we consider the variable $\tilde x=x-\dfrac 13$. The composed function $W_1(1/3+\tilde x)=\tilde{W}_1(\tilde x)$ is still hyperholomorphic, $\tilde{W}_1(0)=0$ and we can consider its expansion at the origin
\[
\tilde{W}_1(\tilde x)= \sum_{\substack{\alpha\in\mathbb{N}_0^3\\ \alpha\not=(0,0,0)}} {\zeta}(\tilde{x})^\alpha f_{\alpha} .
\]
By Lemma \ref{fog} we have that for any $\rho >0$ there exists $\epsilon >0$ such that for $\tilde{x}_0^2+\tilde{x}_j^2 <\epsilon$, $j=1,2,3$, one has
\[
\left| \left(\tilde{W}_1(\tilde{x})\right)^{\odot n}\right| < \rho^n, \qquad n=1,2,\ldots.
\]
Since $\tilde{W}_n(\tilde{x})=W_n(1/3+x)=(W_1(1/3+x))^{\odot n}=(\tilde{W}_1(\tilde{x})^{\odot n}$ we deduce that $|\tilde{W}_n(\tilde x)|<\rho^n$ for $\tilde{x}_0^2+\tilde{x}_j^2 <\epsilon$, $j=1,2,3$. Thus, for any $0<\rho<1$ there exists $\epsilon >0$ such that
\[
\sum_{n=0}^\infty |W_n(x)|^2 < \sum_{n=0}^\infty \rho^{2n}
\]
for $(x_0-1/3)^2+x_j^2<\epsilon$, $j=1,2,3$ and the statement follows.
\end{proof}
\begin{defn}
We denote by $\mathcal P$ the subset of $\mathbb H$  for which the series
\begin{equation}
\label{pi-r}
\sum_{n=0}^\infty |W_n(x)|^2
\end{equation}
 converges.
\end{defn}
The set $\mathcal{P}$ is nonempty as the previous lemma shows, moreover we have:
\begin{prop}
The $\mathcal P$ contains all the points of the positive real axis $\mathbb{R}^+$ and for any $\tilde{x}_0\in\mathbb{R}^+$ there exists a neighborhood of $\tilde{x}_0$ in which the series converges.
\end{prop}
\begin{proof}
Let us consider a point $\tilde{x}_0\in\mathbb{R}^+$, then $W_1(\tilde{x}_0)=w_1\in\mathbb{R}$ with $|w_1|<1$.
Let $\tilde{x}=x-\tilde{x}_0$ and set $W_1(\tilde{x}+\tilde{x}_0)= \breve{W}_1(\tilde{x})$. Since
$\breve{W}_1(0)=w_1$ we write $\breve{W}_1(\tilde{x})=w_1+\tilde{W}_1(x)$ with $\tilde{W}_1(0)=0$.
Let us set $\breve{W}_n(\tilde{x})=(\breve{W}_1(\tilde{x}))^{\odot n}$ then
\[
\breve{W}_n(\tilde{x})=(\breve{W}_1(\tilde{x}))^{\odot n}=(w_1+\tilde{W}_1(\tilde{x}))^{\odot n}=\sum_{k=0}^n {n \choose k} \tilde{W}_1(\tilde{x})^{\odot k} w_1^{n-k},
\]
so that
\[
|\breve{W}_n(\tilde{x})|\leq \sum_{k=0}^n {n \choose k} |\tilde{W}_1(\tilde{x})|^{k} |w_1|^{n-k}.
\]
The proof of Lemma \ref{Wenne} shows that for any $0<\rho<1$ there exists $\epsilon >0$ such that
$|\tilde{W}_1(\tilde{x})|^{k}<\rho^k$ for $(x-\tilde{x}_0)^2+x_j^2<\epsilon$, $j=1,2,3$. Setting $\eta=1-|w_1|>0$, we take $\rho=\eta/2$ so that we have
\[
|\breve{W}_n(\tilde{x})|\leq (|w_1|+\rho)^n=(1-\eta/2)^n
\]
and $|W_n(x)|^2<(1-\eta/2)^n$ in a neighborhood of $\tilde{x}_0$.
The assertion follows.
\end{proof}

As in Section \ref{S2}, $K_W(x,y)$ solves the equation \eqref{lurking}, with $W_1$ replacing $P_1$, namely
\begin{equation}
\label{eqkw}
K_W(x,y)-W_1(x)\odot K_W(x,y)\odot_r \overline{W_1(y)}=1.
\end{equation}
We set (note that we do not put a factor $2\pi$):
\begin{equation}
K_{\mathcal P}(x,y)= (1+P_1(x))^{-\odot}\odot K_W(x,y)\odot_r(1+\overline{P_1(y)})^{-\odot_r}.
\label{khp}
\end{equation}
\begin{thm}
The kernel $K_{\mathcal P}(x,y)$ is positive definite in $\mathcal P$ and is the unique solution of the Lyapunov equation
\begin{equation}
\label{lyapunov}
2(P_1(x)\odot K_{\mathcal P}(x,y)+K_{\mathcal P}(x,y)\odot_r\overline{P_1(y)})=1.
\end{equation}
\end{thm}

\begin{proof}
The first claim follows from the formula
\begin{equation}
\label{kp-expansion}
K_{\mathcal P}(x,y)=\sum_{n=0}^\infty\left(1+P_1(x))^{-\odot}\odot W_n(x)\right)\left(1+P_1(y))^{-\odot_r}\odot_r W_n(y)\right)^*
\end{equation}
where we have used Proposition \ref{anti} relating the left and right $CK$-products.\\

We have from \eqref{khp}
\begin{equation}
K_W(x,y)=(1+P_1(x))\odot K_{\mathcal P}(x,y)\odot_r(1+\overline{P_1(y)}).
\end{equation}
Moreover, we have
\begin{equation}
 \label{p1p1p1}
(1+P_1(x))\odot(1-P_1(x))\odot(1+P_1(x))^{-\odot}=(1-P_1(x)).
\end{equation}
Using these equations we rewrite \eqref{eqkw} as
\begin{equation}
\label{better}
(1+P_1(x))\odot K_{\mathcal P}(x,y)\odot_r(1+\overline{P_1(y)})-(1-P_1(x))\odot K_{\mathcal P}(x,y)\odot_r(1+\overline{P_1(y)})=1
\end{equation}
from which we get \eqref{lyapunov}.
\end{proof}

We denote by $\mathbf H_2(\mathcal P)$ the reproducing kernel Hilbert space with reproducing kernel equal to $K_{\mathcal P}(x,y)$. We also note that \eqref{lyapunov} and \eqref{better} are equivalent,
but \eqref{better} is better adapted to use the lurking isometry method or the linear relation method, when one considers multipliers (see Section \ref{mult123}).\\

The counterpart of the expansion \eqref{hoffman} is presented in the following theorem.

\begin{thm}
The reproducing kernel Hilbert space associated with the kernel $K_{\mathcal P}(x,y)$ consists of the power series
\begin{equation}
f(x)=\sum_{n=0}^\infty (1+P_1(x))^{-\odot}  \odot W_n(x)f_n,
\end{equation}
where the coefficients $f_n$ are in $\mathbb H$ and satisfy $\sum_{n=0}^\infty |f_n|^2<\infty$. The latter is then the square of the norm of $f$.
\end{thm}

\begin{proof}
This follows from \eqref{kp-expansion}.
\end{proof}

We now turn to the characterization of $\mathbf H_2(\mathcal P)$ in terms of restrictions to the real positive axis.

\begin{thm}
  $f\in\mathbf H_2(\mathcal P)$ if and only if $x_0\mapsto \chi(f(3x_0))$ is the restriction to $x_0>0$ of an element in $\left(\mathbf H_2(\mathbb C_r)\right)^{2\times 2}$. The map
  which to $f\in\mathbf H_2(\mathcal P)$ associates the map $x_0\mapsto \sqrt{\pi}\chi(f(x_0))$ is then unitary.
\end{thm}

\begin{proof}
Setting $x_1=x_2=x_3=0$ we get
\[
f(x_0)=\sum_{n=0}^\infty\frac{(1-3x_0)^n}{(1+3x_0)^{n+1}}f_n.
\]
Applying the map $\chi$ and comparing with \eqref{hoffman} we get the statement. The function is uniquely determined since $(0,\infty)$ is a zero set. The converse follows from the fact that
$x_0\mapsto\dfrac{(1-3x_0)^n}{(1+3x_0)^{n+1}}$ has as unique axially hyperholomorphic extension the function $(1+P_1(x))^{-\odot}\odot W_n(x)$ (which is evidently of axial type),
being the $CK$-product of series in $P_1(x)$.
\end{proof}

\section{Schur multipliers in the half-plane setting}
\setcounter{equation}{0}
\label{mult123}
In the classical case of the complex numbers, a $\mathbb C^{r\times r}$-valued function is contractive in $\mathbb C_r$ if and only if the kernel
\begin{equation}
  \label{ks123321}
\frac{I_r-s(z)s(w)^*}{z+\overline{w}}
\end{equation}
is positive definite in $\mathbb C_r$. More generally, if a function $s$ is {\sl defined} in a zero set, say $\mathcal Z$, of the open right half-plane and the kernel \eqref{ks123321} is positive
definite on $\mathcal Z$, then $s$ is the restriction to $\mathcal Z$ of a uniquely defined function analytic and contractive in $\mathbb C_r$. This can be seen from the disk case
(see Theorem \ref{dono}) using a Cayley transform. In the present section we study the counterpart of the Schur multipliers for the space $\mathbf H_2(\mathcal P)$, and characterize them in three
equivalent ways:\\
$(1)$ In terms of a positive definite kernel.\\
$(2)$ In terms of an appropriately defined multiplication operator.\\
$(3)$ In terms of a realization.\\

\begin{defn}
A $\mathbb H^{r\times r}$-valued $S$ function is called a Schur multiplier if there is a kernel $K_S(x,y)$ positive definite in $\mathcal P$,
left-hyperholomophic in $x$ and right-hyperholomorphic in $y$, and such that
\begin{equation}
\label{lyapunov!!!}
2(P_1(x)\odot K_S(x,y)  +K_S(x,y)\odot_r\overline{P_1(y)})=I_r-S(x)S(y)^*,\quad x,y\in\mathcal P.
\end{equation}
\end{defn}

We will get a description of all such multipliers in terms of a realization, but already mention a very easy example, namely $S(x)=W_1(x)$. Then,
with
\[
K_S(x,y)=(1+P_1(x))^{-\odot}(1+\overline{P_1(y)})^{-\odot_r}
\]
we have
\[
2(P_1(x)\odot K_S(x,y)+K_S(x,y)\odot_r \overline{P_1(y)})=1-W_1(x)\overline{W_1(y)},\quad x,y\in\mathcal P.
\]

In view of the lurking isometry method, it is better to write \eqref{lyapunov!!!}
as
\begin{equation}
  \begin{split}
    (1+P_1(x))\odot K_S(x,y)\odot_r(1+\overline{P_1(y)})-(1-P_1(x))\odot K_S(x,y)\odot_r(1-\overline{P_1(y)})&\\
    &\hspace{-3cm}= I_r-S(x)S(y)^*
\end{split}
  \end{equation}
or as
  \begin{equation}
    \begin{split}
      (1+P_1(x))\odot K_S(x,y)\odot_r(1+\overline{P_1(y)})+S(x)S(y)^*&=\\
     & \hspace{-3cm}=(1-P_1(x))\odot K_S(x,y)\odot_r(1-\overline{P_1(y)})+  I_r.
\end{split}
    \end{equation}

\begin{thm} The $\mathbb H^{r\times s}$-valued function $S$ is a Schur multiplier if and only if there exist a Hilbert space $\mathcal H$ and a co-isometry
  \begin{equation}
    \begin{pmatrix}A&C\\ B&D\end{pmatrix}
  \end{equation}
  such that
\begin{equation}
  S(3x_0)=D+\frac{1-3x_0}{1+3x_0}C\left(I-\frac{1-3x_0}{1+3x_0}A\right)^{-1}B,\quad x_0\in(-1/3,1/3),
\label{schur123321}
\end{equation}
with unique hyperholomorphic extension (of axial type) to $\mathcal P$ given by
\begin{equation}
S(x)=D+\sum_{n=0}^\infty W_n(x)CA^nB.
\end{equation}
\label{thh2}
\end{thm}

\begin{proof}
Write
\[
  K_S(x,y)=X(x)X(y)^*
\]
where $X$ is operator-valued and hyperholomorphic of axial type (for instance, via the associated reproducing kernel Hilbert space; one takes $X(x)$ to be the point evaluation at $x$).
We get the isometric relation (the lurking isometry) spanned by the pairs
\[
\left(   \begin{pmatrix}X(3y_0)^*(1-\overline{P_1(3y_0)})h\\
   h\end{pmatrix},\begin{pmatrix}X(3y_0)^*(1+\overline{P_1(3y_0)})h\\
    S(3y_0)^*h\end{pmatrix}
\right)
\]
with $y_0\in(-1/3,1/3)$ and $h\in\mathbb H^s$.
Write the isometry as
\begin{equation}
  \label{abcd}
\begin{pmatrix}A^*&B^*\\C^*&D^*\end{pmatrix}.
\end{equation}
We get
\[
\begin{split}
\frac{1-3y_0}{1+3y_0}A^*X(y_0)^*h+\frac{1}{1+3y_0}C^*h&=X(y_0)^*h\\
(1-3y_0)B^*h+\sqrt{2}D^*h&=S(y_0)^*h
\end{split}
\]
Hence,
\[
S(3y_0)^*=D^*+\frac{1-3y_0}{1+3y_0}B^*\left(I-\frac{1-3y_0}{1+3y_0}A^*\right)^{-1}C^*
\]
and hence the result, since $W_n(x)$ is the unique hyperholomorphic extension of axial type.
\end{proof}

In the complex setting case,  a function, say $s$, analytic and contractive in $\mathbb C_r$ does not need belong to $\mathbf H^2(\mathbb C_r)$, but $z\mapsto s(z)/(1+z)$ does belong to
$\mathbf H^2(\mathbb C_r)$. Here, at least in the present analysis, we need a supplementary condition to get the counterpart of this result.
We have (the notion of spectral radius is defined for a quaternionic operator $A$ as in
the classical case by the formula $\rho(A)=\limsup_{n\rightarrow\infty}\|A^n\|^{1/n}$):

\begin{cor} In the notation of Theorem  \ref{thh2}, assume $\rho(A)<1$. Then, the entries of the function $(1+P_1(x))^{-\odot}\odot S(x)$ belong to $\mathbf H_2(\mathcal P)$.
\end{cor}

\begin{proof}
This follows from the fact that
\[
\|CA^nB\|\le \|C\|\cdot\|B\|\cdot \|A^n\|,
\]
and the series $\sum_{n=1}^\infty \|A^n\|$ converges since $\rho(A)<1$.
\end{proof}

\begin{thm}
  $S$ is a Schur multiplier if and only if the operator defined by
  \[
    M_Sf=\sum_{n=0}^\infty (1+P_1(x))^{-\odot}(W_n\odot S)(x)f_n
  \]
  is a contraction from $\mathbf H_2(\mathcal P)$ into itself, and $K_S$ is given by
  \begin{equation}
    \label{ks1234}
    \begin{split}
      K_S(x,y)&=(1+P_1(x))^{-\odot}\odot\\
      &\hspace{5mm}\odot\left(\sum_{n=0}^\infty W_n(x)\overline{W_n(y)}-\sum_{n=0}^\infty (W_n\odot S)(x)\overline{(W_n\odot S)(y)}\right)\\&  \hspace{7cm}\odot_r(1+\overline{P_1(y)})^{-\odot_r}.
      \end{split}
    \end{equation}
\end{thm}

\begin{proof} We consider the scalar case to ease the notation and first remark that, if $S$ is a Schur multiplier we have
  \[
    K_S(3x_0,3y_0)=\frac{1-S(3x_0)\overline{S(3y_0)}}{3(x_0+y_0)}
  \]
  with axially hyperholomorphic extension \eqref{ks1234}, and the positivity of \eqref{ks1234} expresses that the operator $M_S$ is a contraction. The converse is proved
  by defining a contractive relation from the positivity of the kernel, and show that the relation extends to the graph of $M_S^*$.
\end{proof}

As in the $\mathcal E$ setting, the  case where the isometry in the above realizations is unitary in a finite dimensional space corresponds to finite dimensional $\mathcal H(S)$ spaces isometrically
included in the Hardy space $(\mathbf H_2(\mathcal P))^r$. When the space has dimension $1$ the function $S$ is the counterpart of a Blaschke factor of the half-plane.
\begin{thm}
  The space $\mathcal H(S)$ is finite dimensional if and only if the space $\mathfrak H$ can be chosen finite dimensional. When $r=s$, the space $\mathcal H(S)$ is isometrically included inside
  $(\mathbf H_2(\mathcal P))^r$.
\end{thm}
\begin{proof} We set
  \[
    \begin{pmatrix} \mathsf A&\mathsf B\\
      \mathsf C&\mathsf D\end{pmatrix}
  \]
  the matrix of the underlying unitary map. We have:
  \[
    \begin{split}
      M_S\left((1+P_1(x))^{-\odot}\odot W_m(x)h  \right)&=\\
     & \hspace{-4cm}=(1+P_1(x))^{-\odot}\odot W_m(x)Dh+\sum_{n=0}^\infty (1+P_1(x))^{-\odot}\odot W_{m+n}(x)\mathsf C\mathsf A^n\mathsf Bh.
          \end{split}
        \]

        The same computations as in the proof of Proposition \ref{propo-iso} show that
        \[
          \begin{split}
          \langle M_S\left(1+P_1(x))^{-\odot}\odot W_{n_1}h(x)\right)  ,
          M_S\left(1+P_1(x))^{-\odot}\odot W_{n_2}(x)k\right)  \rangle&=\\
          &\hspace{-3cm}=\sum_{\substack{n,m=0\\
            n_1+n=n_2+m}}^\infty k^*\mathsf  B^*\mathsf  A^{*n}\mathsf  A\mathsf C^*\mathsf C^m\mathsf Bh\\
          &\hspace{-3cm}=          \delta_{n,m}k^*h,\quad h,k\in\mathbb H^s,
          \end{split}
        \]
        and this allows to end the proof.
\end{proof}
\section{Carath\'eodory multipliers in the half-plane setting}
\setcounter{equation}{0}
\label{carasec}
A function $\Phi$ analytic and with a positive real part in the open right-half plane is called a Carath\'eodory function, and is
 characterized by the positivity of the kernel
\[
\frac{\Phi(z)+\overline{\Phi(w)}}{z+\overline{w}}
\]
in $\mathbb C_r$. As for Herglotz functions, a Carath\'eodory function need not be a multiplier of the Hardy space $\mathbf H_2(\mathbb C_r)$. We now introduce the counterpart of this class of
functions in the present setting. As in Section \ref{herg} we use the term {\sl multiplier} (rather than, for instance {\sl pseudo-multiplier}) although the $CK$-multiplication by the function $\Phi$
is not assumed bounded.

\begin{defn}
An $\mathbb H^{r\times r}$-valued function $\Phi$ is called a Carath\'eodory pseudo-multiplier if there is a kernel $K_\Phi(x,y)$ positive definite in $\mathcal P$,
left-hyper-holomophic in $x$ and right-hyperholomorphic in $y$, and such that
\begin{equation}
  \label{phi789}
2(P_1(x)\odot K_\Phi(x,y)  +K_\Phi(x,y)\odot_r\overline{P_1(y)})=\Phi(x)+\Phi(y)^*.
\end{equation}
\end{defn}

A first example is given by $\Phi(x)=aP_1(x)$ with $a>0$ and $K_\Phi(x,y)=a/2$. It follows from the definition that a sum of Carath\'eodory multipliers is
a Carath\'eodory multiplier, and so is $\Phi^{-\odot}$ and $a\Phi$ with $a>0$. Therefore, any sum of the form
\[
\Phi(x)=a_0P_1(x)+\sum_{n=1}^Nb_n(a_n+P_1(x))^{-\odot}
\]
is a Carath\'eodory multiplier for every choice of $a_0\ge 0$, $a_1,\ldots, a_N>0$ and $b_1,\ldots, b_N\ge 0$.\\

It is more convenient to rewrite \eqref{phi789} as
\begin{equation}
  \begin{split}
    (1+P_1(x))\odot K_{\Phi}(x,y)\odot_r(1+\overline{P_1(y)})-(1-P_1(x))\odot K_{\Phi}(x,y)\odot(1-\overline{P_1(y)}&\\
    &\hspace{-8cm}=(I_r+\Phi(x))(I_r+\Phi(y)^*)-(I_r-\Phi(x))(I_r-\Phi(y)^*).
\end{split}
\end{equation}

\begin{thm}
    $\Phi$ is a Carath\'eodory multiplier if and only if it can be written as
  \begin{equation}
    \Phi(3y_0)=(I_r-S(3y_0))(I_r+S(3y_0))^{-1}
  \end{equation}
  with $S$ as in \eqref{schur123321}.
\end{thm}

\begin{proof}
Write
\[
  K_\Phi(x,y)=X(x)X(y)^*
\]
where $X(x)$ is the point evaluation in the reproducing kernel Hilbert space with reproducing kernel $K_\Phi$.
We get the isometric relation (the lurking isometry) defined by the right linear span of the pairs
\[
\left(  \begin{pmatrix}X(3y_0)^*(1-\overline{P_1(3y_0)})h\\
    (I_r+\Phi(3y_0)^*)h\end{pmatrix},
 \begin{pmatrix}X(3y_0)^*(1+\overline{P_1(3y_0)})h\\
    (I_r-\Phi(3y_0)^*)h\end{pmatrix}\right),
\]
with $y_0\in(-1/3,1/3)$ and $h\in\mathbb H^r$.
Furthermore, with the same notation as \eqref{abcd},
\[
\begin{split}
\frac{1-3y_0}{1+3y_0}A^*X(3y_0)^*h+\frac{1}{1+3y_0}C^*(I_r+\Phi(3y_0)^*)h&=X^*(3y_0)h,\\
(1-3y_0)B^*X^*(3y_0)h    +D^*((I_r+\Phi(3y_0)^*)h&=(I-\Phi(3y_0)^*)h,
 \end{split}
\]
and hence
\[
\begin{split}
(I_r-\Phi(3y_0)^*)h-D^*(I_r+\Phi(3y_0))^*h=\frac{1-3y_0}{1+3y_0}B^*
\left(I-\frac{1-3y_0}{1+3y_0}A^*\right)^{-1}\\& \hspace{-2.5cm}\times C^*(I_r+\Phi(3y_0))^*h,
\end{split}
\]
so that, with $S(3y_0)$ as in the previous theorem,
\[
I_r-\Phi(3y_0)^*=S(3y_0)^*(I_r+\Phi(3y_0)^*)
\]

and hence the result.
\end{proof}

Here too, the space $\mathcal L(\Phi)$ will be finite dimensional if and only if the space $\mathfrak H$ can be chosen finite dimensional.

\section{A table}
We conclude this paper with a table comparing the slice hyperholomorphic case, the case of Fueter variables and the present setting.
\setcounter{equation}{0}

\setcounter{equation}{0}
\newpage
\begin{tabular}[H]{|c|c|c|c|}
  \hline
  &  & &\\
  & Slice setting&Appell setting&General $CK$-setting\\
  & & &\\
  \hline
  & & &\\
  variable &$p\in\mathbb H$&$P_1(x)=CK(x_1{\mathbf e}_1+x_2{\mathbf e}_2+x_3{\mathbf e}_3)$&$\hspace{-3cm}\zeta(x)=$\\
  & &$\hspace{4mm}=\zeta_1(x){\mathbf e}_1+\zeta_2(x){\mathbf e}_2+\zeta_3(x){\mathbf e}_3$ &$=\begin{pmatrix}\zeta_1(x)&\zeta_2(x)&\zeta_3(x)\end{pmatrix}$\\
  & & &\\
  \hline& & &\\
  ``unit disk''&$x_0^2+x_1^2+x_2^2+x_3^2<1$&$9x_0^2+x_1^2+x_2^2+x_3^2<1$&
                                                                          $3x_0^2+x_1^2+x_2^2+x_3^2<1$\\
  & & &\\
  \hline
  & & &\\
  product& $\star$-product&$CK$-product&$CK$-product\\
  & & &\\
  \hline
  & & &\\
  & & &\\
 properties & stays inside the space &{\em outside the space} &stays inside the space\\
  & &{\em Not a power series in $P_n$} &\\
  & & &\\
  \hline
  & & &\\
  &$(I-pA)^{-\star}=$&$(I-P_1A)^{-\odot}$
       &$(I-\zeta A)^{-\odot}=$ \\
 inverse & & &\\
  &$=\sum_{n=0}^\infty p^nA^n$ & {\em outside the space} &   $=\sum_{\alpha\in\mathbb N_0^3} \frac{|\alpha!|}{\alpha!}\zeta^{\alpha}A^\alpha$  \\
  & &{\em Not a power series in $P_n$} &\\
  \hline
 Hardy or & &
       &$k_y(x)=$\\
  Drury-Arveson&$k(p,q)=\sum_{n=0}^\infty p^n\overline{q}^n$ &$\sum_{n=0}^\infty P_n(x)\overline{P_n(y)}$ &  $=\sum_{\alpha\in\mathbb N_0^3}
  \frac{|\alpha!|}{\alpha!}\zeta^{\alpha}(x)\overline{\zeta^\alpha(y)}$  \\  reproducing kernel & & &\\
  \hline
  & & &\\
  structural identity&$I-M_pM_p^*=C^*C$ &$I-M_{P_1}M_{P_1}^*=C^*C$&$I-M_{\zeta}\mathcal M_{\zeta}^*=C^*C$\\
               &    &(does not hold in $Q_n$ setting)  &                                \\
  \hline
  & & &\\
  multiplication&Cauchy product on &$M_S\left(\sum_{n=0}^\infty P_nu_n\right)=$
       & $\odot$-multiplication\\
  operator&coefficients &$=\sum_{n=0}^\infty (P_n\odot S)u_n$ &\\
  & & &\\
  \hline
  adjoint of & & &\\
  multiplication&$M_S^*k(p,q)=$  &$M_S^*\left(\sum_{n=0}^\infty P_n\overline{P_n(a)}u\right)=$
       &$M_S^*k_y(x)=$\\
  operator&$=\sum_{n=0}^\infty p^n\overline{S(q)}\overline{q}^n$ &$=\sum_{n=0}^\infty P_n\overline{(P_n\odot S)(a)}u$ &$=\sum_{\alpha\in\mathbb N_0^3}\frac{|\alpha|!}{\alpha!}
                                                                                                                        \zeta^\alpha(x)\cdot$\\
  & & &$\cdot\overline{(S\odot\zeta^\alpha(y))}$\\
  \hline
  & & &\\
        backward shift &  kernel eigenvector&kernel not eigenvector&kernel common\\ operator& & & eigenvector of the shifts\\
  \hline
  & & &\\
  rational functions& ring&group&ring\\
  & & &\\
  \hline
\end{tabular}

\begin{rem}{\rm
For the formulas for the adjoint operator, see e.g.
\cite[(3.9), p. 160]{2013arXiv1308.2658A} for the slice hyperholomorphic case and
\cite[Proposition 2.2. p. 34]{MR2240272} for the Fueter series setting.}
\end{rem}




\section*{Acknowledgment}
Daniel Alpay thanks the Foster G. and Mary McGaw Professorship in Mathematical Sciences, which supported this research. Kamal Diki acknowledges the support of the project INdAM Doctoral Programme in Mathematics and/or Applications Cofunded by Marie Sklodowska-Curie Actions, acronym: INdAM-DP-COFUND-2015, grant number: 713485.


\end{document}